\documentclass[reqno]{amsart}

\usepackage{macros}

\toggletrue{fact}

\counterwithin{theorem}{subsection}

\makeatletter
\def\enumfix{%
\if@inlabel
 \noindent \par\nobreak\vskip-\topsep\hrule\@height\z@
\fi}

\let\olditemize\itemize
\def\itemize{\enumfix\olditemize}

\makeatother

\begin{document}

\title{Factorization homology of enriched $\infty$-categories}

\author{David Ayala, John Francis, Aaron Mazel-Gee, and Nick Rozenblyum}

\date{\today}

\begin{abstract}
For an arbitrary symmetric monoidal $\infty$-category $\cV$, we define the factorization homology of $\cV$-enriched $(\infty,1)$-categories over (possibly stratified) 1-manifolds and study some of its basic properties.  
In the case of spectral enrichment, we show that the value of factorization homology on a circle is topological Hochschild homology.  
\end{abstract}

\thanks{DA gratefully acknowledges the support of the NSF under awards 1507704, 1812055, and 1945639. JF gratefully acknowledges the support of the NSF under awards 1508040 and 1812057. AMG gratefully acknowledges the support of the NSF Graduate Research Fellowship Program (grant DGE-1106400) as well as the hospitality of Montana State University.}

\maketitle

\setcounter{tocdepth}{2}
\tableofcontents
\setcounter{tocdepth}{2}

\setcounter{section}{-1}

\section{Introduction}
\label{section.intro}

\startcontents[sections]


\subsection{Categorified integration in quantum field theory}

A fundamental link between manifolds and higher algebraic structures is \bit{factorization homology}: in its most primitive form, this takes a framed $n$-manifold $M$ and an $\EE_n$-algebra $A$ -- in chain complexes, say -- and returns a chain complex
\[
\int_M A
\]
obtained by ``integrating'' the algebra $A$ over configurations of $n$-disks in $M$ \cite{LurieHA,AF-oldfact}.\footnote{One can weaken the condition that $M$ be framed in exchange for requiring further structure on $A$ \cite{AF-oldfact}.}  Besides being intimately related to the study of mapping spaces and proximate notions in manifold topology \cite{Salv-conf,Kallel-part,Bodig-split,McDuff-conf,May-GILS,Segal-conf,LurieHA,AF-oldfact}, factorization homology is a close cousin of Beilinson--Drinfeld's algebro-geometric theory of \textit{chiral homology} \cite{BD-chiral}, which computes conformal blocks in conformal field theory.  These ideas have since spawned much subsequent activity in mathematical physics, notably Costello--Gwilliam's work on perturbative quantum field theory \cite{CostGwill-fact1,CostGwill-fact2} wherein they recover global observables from local observables via factorization homology.

Higher-dimensional defects are essential for describing field theories -- even mapping spaces -- that are not perturbative.  Towards accommodating such defects, in \cite{AFR-fact} Ayala--Francis--Rozenblyum generalized the algebraic input of factorization homology from $\EE_n$-algebras to \textit{$(\infty,n)$-categories with adjoints} ($n\leq 2$):\footnote{Just as an algebra determines a one-object category, an $\EE_n$-algebra determines an $(\infty,n)$-category with a unique $k$-morphism for all $k < n$.} for a framed $n$-manifold $M$ and an $(\infty,n)$-category $\cC$ with adjoints, they defined the factorization homology
\[
\int_M \cC
\]
of $\cC$ over $M$ as the space of labelings by $\cC$ of \textit{disk-stratifications} of $M$.  This construction allows for field theories that are not necessarily determined by their point-local observables.

However, this construction is still one step removed from the production of TQFTs of physical interest, which are \textit{linear} in nature: where the framework of \cite{AFR-fact} yields a space, one would like to obtain a vector space or chain complex.  In this paper, we provide a blueprint for the appropriate generalization: we construct the factorization homology of \textit{enriched} $(\infty,1)$-categories.
We expect that our construction contains all the essential features of a full theory of factorization homology of enriched $(\infty,n)$-categories with adjoints.
The key idea that drives our approach can be summarized as follows.

\begin{slogan}
\label{main.fact.slogan}
\bit{Enriched} factorization homology arises from \bit{categorified} factorization homology.
\end{slogan}

\noindent We will explain \Cref{main.fact.slogan} in \Cref{subsection.fact.hlgy.intro}.

\subsection{Hochschild homology, cyclic homology, and the cyclotomic trace}

Our primary motivation for constructing enriched factorization homology comes from a different direction, namely topological Hochschild homology and its connection with algebraic K-theory.

Recall that the \textit{Hochschild homology} of an associative ring $A$ 
is the value
\[
\HochH(A) \in \mathbf{D} (\Mod_\ZZ)
\]
of the left derived functor of the coinvariants functor
\[ \begin{tikzcd}[row sep=0cm]
\BiMod_{(A,A)}
\arrow{r}
&
\Mod_\ZZ
\\
\rotatebox{90}{$\in$}
&
\rotatebox{90}{$\in$}
\\
M
\arrow[maps to]{r}
&
M/[A,M]
\end{tikzcd} \]
at the $(A,A)$-bimodule $A$. 
This can be identified with the factorization homology over $\SS^1$
\[
\HochH(A)
\simeq
\int_{\SS^1} A
\]
of $A$, considered as an $\EE_1$-algebra in the $\infty$-category $\Mod_{H\ZZ}$ of $H\ZZ$-modules (i.e., chain complexes localized at the quasi-isomorphisms) \cite{LurieHA}.  
In fact, this equivalence lifts to a direct identification of the cyclic bar construction as computing factorization homology over $\SS^1$ (see \Cref{r7}\Cref{example:circle}).

From our perspective, a crucial advantage of the definition of Hochschild homology via factorization homology is that
\begin{itemize}
\item[] \textit{factorization homology manifests the inherent symmetries of Hochschild homology.}
\end{itemize}
For example, its natural action of the circle group $\TT$ arises simply from the functoriality of factorization homology for automorphisms of framed manifolds.  By contrast, this action only appears through the machinery of cyclic sets (or mixed complexes) when Hochschild homology is defined through simplicial (or homological) methods \cite{Loday-cyclic}.
This passage from handicrafted functoriality to manifest functoriality is analogous to the passage from cellular homology to singular homology.

This $\TT$-action on Hochschild homology is fundamental in the study of the \bit{algebraic K-theory} of $A$.  Namely, the latter admits a homotopy $\TT$-invariant \textit{Dennis trace} map to $\HochH(A)$, 
and Goodwillie proved \cite{GooRelAKT} that the resulting \textit{cyclic trace} map
\[ \begin{tikzcd}
\K(A)
\arrow{rr}
\arrow[dashed]{rd}
&
&
\HochH(A)
\\
&
\HC(A)
\arrow{ru}
&
\hspace{-38pt} := \HochH(A)^{\htpy \TT}
\end{tikzcd} \]
to \textit{negative cyclic homology} is ``locally constant'' after rationalization.

Seeing that his theorem could not hold integrally, Goodwillie conjectured the existence of a refinement of $\HochH$ obtained by replacing the ground ring $\ZZ$ with the sphere spectrum $\SS$.  This was cleverly constructed by B\"okstedt \cite{Bok-THH}, despite an utter lack of sufficient foundations (namely a good model category of spectra): for an arbitrary associative (i.e., $\EE_1$-)ring spectrum $A$, he defined its \bit{topological Hochschild homology} spectrum
\[
\THH(A)
\in \Spectra
~.
\]
He moreover endowed this with a homotopy $\TT$-action and a homotopy $\TT$-invariant (``topological'') Dennis trace map
\[
\K(A)
\longra
\THH(A)
~,
\]
which factors the Dennis trace map to $\HochH(A)$ in the case that $A$ is an ordinary ring.

But the homotopy $\TT$-fixedpoints of $\THH$
proved to be still insufficient to integrally approximate algebraic K-theory.  In their celebrated paper \cite{BHM}, B\"okstedt--Hsiang--Madsen introduced a further refinement
\[
\TC(A)
:=
\THH(A)^{\htpy \Cyclo}
\]
of $\THH$, which they called \bit{topological cyclic homology}: the $\TT$-action on $\THH(A)$ admits an enhancement to a \bit{cyclotomic structure}, and $\TC(A)$ is defined as its homotopy invariants.  Soon after, Dundas--Goodwillie--McCarthy proved that the resulting factorization
\[ \begin{tikzcd}
\K(A)
\arrow{rr}
\arrow[dashed]{rd}
&
&
\HochH(A)
\\
&
\TC(A)
\arrow{ru}
&
\hspace{-28pt} := \THH(A)^{\htpy \Cyclo}
\end{tikzcd} \]
of the Dennis trace, called the \bit{cyclotomic trace}, is indeed integrally ``locally constant'' \cite{GooRelAKT,McCRelAKT,DundasRelAKT,DM-KTHH,DGM-book}.
Following these breakthroughs, the cyclotomic trace became a central tool: the vast majority of known computations of algebraic K-theory result from this infinitesimal behavior \tracerefs.

This paper is part of a series, whose overarching purpose is to provide a precise conceptual description of the cyclotomic trace at the level of derived algebraic geometry: this is explained in \cite[\S 0]{AMR-trace}.  The series begins with the paper \cite{AMR-cyclo}, in which we reidentify the $\infty$-category of cyclotomic spectra in terms of \textit{naive} equivariant homotopy theory.

In order to explain the role that this paper plays within the series, let us first note that in the context of these various theories ($\K$, $\THH$, and $\TC$), algebro-geometric objects such as schemes and stacks are incarnated through their stable $\infty$-categories of vector bundles (i.e., perfect complexes).\footnote{Indeed, these theories are all Morita invariant: K-theory is so by definition, and the rest are so because $\THH$ is so as a functor to cyclotomic spectra (see \cite[Remarks \ref*{trace:flagged}, \ref*{trace:morita.invce}, \and \ref*{trace:fgt.cyclo.is.conservative}]{AMR-trace}).  Hence, we lose nothing by passing from associative ring spectra to their stable $\infty$-categories of perfect modules.}  As justified through Theorem~\ref{main.theorem} below, it is possible to define $\THH$ 
of an arbitrary spectrally-enriched (e.g.\! stable) $\infty$-category, or more generally of a $\cV$-enriched $(\infty,1)$-category $\cC$ for an arbitrary $\otimes$-presentable symmetric monoidal $\infty$-category $(\cV,\boxtimes)$. 
Namely, for $\sC$ an ordinary category enriched over the ordinary category of spectra, the \bit{topological Hochschild homology} of $\sC$ is the geometric realization of the cyclic bar construction of its cofibrant replacement $\w{\sC}$:
\[
{\sf THH}(\sC)
~:=~
\Bigl |
{\sf Bar}_\bullet^{\sf cyc} 
(\w{\sC})
\Bigr |
~\in 
\Spectra
~.
\]
Because the cyclic bar construction is the restriction of a paracyclic object in $\Spectra$, there is a canonical action $\TT \lacts {\sf THH}(\sC)$.

\begin{maintheorem}
\label{main.theorem}
Let $(\cV,\boxtimes)$ be a $\otimes$-presentable symmetric monoidal $\infty$-category.
For each $\cV$-enriched $(\infty,1)$-category $\cC$ 
factorization homology over the circle defines a $\TT$-module in $\cV$:
\[
\Bigl(
\TT
\lacts
\int_{\SS^1} \cC
\Bigr)
~\in ~
\Mod_{\TT}(\cV)
~.
\]
In the case that $\cV = \Spectra$, and that $\cC$ is presented by an ordinary category $\sC$ enriched over the ordinary category of spectra, there is a canonical identification between $\TT$-module spectra:
\[
\int_{\SS^1} \cC
~\simeq~
{\sf THH}(\sC)
~.
\]

\end{maintheorem}

\subsection{Factorization homology of enriched $(\infty,1)$-categories}
\label{subsection.fact.hlgy.intro}

We now outline our construction of enriched factorization homology, with the goal of explaining \Cref{main.fact.slogan}. Everything described in this subsection will be made rigorous over the course of the present paper.

We base our work in the $\infty$-category of \textit{compact solidly 1-framed stratified stratified spaces} of \cite{AFR-fact}; for simplicity, we just refer to this as the $\infty$-category of \bit{stratified stratified spaces} and denote it by $\M$.  An object in $\M$ is a finite disjoint union of framed circles and finite directed connected graphs; in particular, note that the 0-disk $\DD^0$ and the framed circle $\SS^1$ both define objects of $\M$.  We will not describe the morphisms in $\M$ here, but we note that they include the opposites of (solidly 1-framed) constructible bundles. 
The category $\M$ also admits an explicit combinatorial description by \cite{circle}.

Now, any stratified stratified space $M \in \M$ admits a category $\D(M)$ of \bit{disk-refinements} of $M$.\footnote{A priori $\D(M)$ is an $\infty$-category, but it is not hard to see that it is actually an ordinary category (see~\cite{circle}).}
A disk-refinement of $M$ consists of a configuration of points in the 1-dimensional strata of $M$, such that no circles remain unstratified (i.e., it is a stratification of $M$ by points and open intervals).  Note that a disk-refinement inherits the structure of a finite directed graph, so that it is itself an object in $\M$, namely one that is \textit{disk-stratified}: it contains no unstratified circles.
The morphisms in $\D(M)$ are generated by
\begin{itemize}
\item disappearances of points,
\item anticollisions of points, and
\item isotopies of configurations of points.
\end{itemize}
We note here that when $M$ is itself disk-stratified, then it defines a final object in $\D(M)$.

After laying the necessary foundations, in \Cref{subsection.enr.fact.hlgy} we will define the factorization homology over $M$ of a $\cV$-enriched $(\infty,1)$-category $\cC$ as the colimit
\begin{equation}
\label{heuristic.formula.for.enr.fact.hlgy}
\int_M \cC
:=
\underset{R \in \D(M)}{\colim}
\left(
\underset{R^{(0)} \xlongra{\lambda} \iC}{\colim}
\left(
\underset{e \in R^{(1)}}{\standaloneboxtimes}
\ulhom_\cC ( \lambda(s(e)) , \lambda(t(e)) )
\right)
\right)
~.
\end{equation}
In words, this will be the colimit
\begin{itemize}
\item indexed over
\begin{itemize}
\item disk-refinements $D$ of $M$ and
\item labelings $\lambda$ of the set of vertices $D^{(0)}$ of $D$ by objects of $\cC$
\end{itemize}
\item of the monoidal product in $\cV$
\begin{itemize}
\item indexed over the set of edges $e \in R^{(1)}$ of $D$
\item of the hom-object in $\cC$
\begin{itemize}
\item from the object $\lambda(s(e)) \in \cC$ labeling the source of $e$
\item to the object $\lambda(t(e)) \in \cC$ labeling the target of $e$.
\end{itemize}
\end{itemize}
\end{itemize}
The structure maps as $D \in \D(M)$ varies will be generated by
\begin{itemize}
\item composition in $\cC$ (for disappearances of points),
\item insertion of identity maps in $\cC$ (for anticollisions of points), and
\item equivalences (for isotopies of configurations of points).
\end{itemize}
This definition will simplify considerably when $M$ is itself disk-stratified: the colimit can then be identified as the value
\[
\int_M \cC
\simeq
\underset{M^{(0)} \xlongra{\lambda} \iC}{\colim}
\left(
\underset{e \in M^{(1)}}{\standaloneboxtimes}
\ulhom_\cC ( \lambda(s(e)) , \lambda(t(e)) )
\right)
\]
of the inner expression in formula \Cref{heuristic.formula.for.enr.fact.hlgy} at the final object $D = M \in \D(M)$.

In order to explain \Cref{main.fact.slogan}, we describe an essential ingredient in making the definition \Cref{heuristic.formula.for.enr.fact.hlgy} rigorously, namely
the formalism of \bit{$\cV$-enriched $(\infty,1)$-categories} of \cite{GH-enr}.  
Towards this end, we
first recall two auxiliary definitions.
\begin{itemize}
\item A monoidal $\infty$-category $(\cV,\boxtimes)$ is specified by its \textit{monoidal deloop}
\[ \bDelta^\op \xra{\BV} \Cat~, \]
i.e., its bar construction.  In fact, this is a \textit{category-object} in $\Cat$, i.e., it satisfies the Segal condition (though not necessarily the completeness condition).
\item For a space $X \in \Spaces$, the \textit{codiscrete category-object} on $X$ is the Segal space
\[
\bDelta^\op
\xra{\cd(X)}
\Spaces
\]
whose space of objects is $X$ and whose hom-spaces are all contractible: in simplicial degree $n$, this is the space $X^{\times (n+1)}$.
\end{itemize}
Now, a $\cV$-enriched $(\infty,1)$-category $\cC$ is specified by the following two pieces of data:
\begin{itemize}
\item its \bit{underlying $\infty$-groupoid}, an object $\iC \in \Spaces$, and
\item its \bit{enriched hom functor}, a right-lax functor
\begin{equation}
\label{enr.hom.functor.for.C}
\cdiC
\xra{\ulhom_\cC}
\BV
\end{equation}
between category-objects in $\Cat$.
\end{itemize}
Let us unwind the functor \Cref{enr.hom.functor.for.C}.  In simplicial degree $n$, it is given by
\[ \begin{tikzcd}[row sep=0cm]
(\iC)^{\times (n+1)}
\arrow{r}
&
\cV^{\times n}
\\
\rotatebox{90}{$\in$}
&
\rotatebox{90}{$\in$}
\\
(C_0,\ldots,C_n)
\arrow[mapsto]{r}
&
( \ulhom_\cC(C_0,C_1) , \ldots , \ulhom_\cC(C_{n-1},C_n) )
\end{tikzcd}
~,
\]
and the right-laxness determines the categorical structure maps; for instance, restricting to the indicated morphisms in $\bDelta^\op$, this specifies the diagram
\[ \begin{tikzcd}
{[2]^\circ}
\arrow{d}[swap]{\delta_1}
&
(\iC)^{\times 3}
\arrow{r}
\arrow{d}
&
\cV^{\times 2}
\arrow{d}
\\
{[1]^\circ}
&
(\iC)^{\times 2}
\arrow{r}[pos=0.3, transform canvas={yshift=0.4cm}]{\rotatebox{-135}{$\Rightarrow$}}[swap, pos=0.3, transform canvas={yshift=-0.4cm}]{\rotatebox{135}{$\Rightarrow$}}
&
\cV
\\
{[0]^\circ}
\arrow{u}{\sigma_0}
&
\iC
\arrow{r}
\arrow{u}
&
\pt
\arrow{u}
\\
\bDelta^\op
&
\cdiC
\arrow{r}{\ulhom_\cC}
&
\BV
\end{tikzcd}
~,
\]
in which the upper square selects the composition maps
\begin{equation}
\label{composition.in.enr.cat}
\ulhom_\cC(C_0,C_1) \boxtimes \ulhom_\cC(C_1,C_2)
\longra
\ulhom_\cC(C_0,C_2)
\end{equation}
while the lower square selects the unit maps
\begin{equation}
\label{unit.in.enr.cat}
\uno_\cV
\longra
\ulhom_\cC(C_0,C_0)
~.
\end{equation}

And now appears \bit{categorified factorization homology}.  Namely, we can take the factorization homology \textit{of category-objects} in $\Cat$: this follows exactly the same prescription given above, only there is now an \textit{$\infty$-category} (instead of just an $\infty$-groupoid) of objects with which to label the vertices of the disk-refinement.  In particular, over a disk-stratified stratified space $D \in \M$, this is easy to describe since $D \in \D(D)$ is a final object, and it is made easier still because $\Cat$ is \textit{cartesian} symmetric monoidal.  Namely, the $\infty$-category
\[
\int_R \cY
\in
\Cat
\]
is just a limit of copies of $\cY_{|[0]^\circ}$ and $\cY_{|[1]^\circ}$: one copy of $\cY_{|[0]^\circ}$ for each vertex, one copy of $\cY_{|[1]^\circ}$ for each edge, and with structure maps $s = \delta_1$ and $t = \delta_0$ determined by the incidence data of the directed graph $D$.

Now, for an arbitrary stratified stratified space $M \in \M$ and a category-object $\cY \in \Fun(\bDelta^\op,\Cat)$, categorified factorization homology defines a functor
\[
\D(M)
\xra{\int_{(-)} \cY}
\Cat
~,
\]
or equivalently a cocartesian fibration
\begin{equation}
\label{catfied.fact.hlgy.of.Y.restr.to.D.of.M}
\begin{tikzcd}
\dint_{\! \! \! |\D(M)} \cY
\arrow{d}
\\
\D(M)
\end{tikzcd}
~.
\end{equation}
This construction is suitably functorial for right-lax functors between category-objects, so that a $\cV$-enriched $(\infty,1)$-category $\cC$ -- that is, its enriched hom functor \Cref{enr.hom.functor.for.C} -- determines a functor
\begin{equation}
\label{catfied.fact.hlgy.of.enr.hom.functor}
\int_{|\D(M)} \cdiC
\xra{\int_{|\D(M)} \ulhom_\cC}
\int_{|\D(M)} \BV
~.
\end{equation}
Over a disk-refinement $D \in \D(M)$, an object in its source is simply given by a labeling of its set of vertices $D^{(0)}$ by objects of the $\infty$-groupoid
\[ \iC =: \cdiC_{|[0]^\circ} \]
(since the morphism-data in $\cdiC$ are canonically determined), while an object in its target is simply given by a labeling of its set of edges $D^{(1)}$ by objects of the $\infty$-category
\[ \cV =: \BV_{|[1]^\circ} \]
(since the object-data in $\BV$ are canonically determined).  In other words, over this object in $\D(M)$ the map \Cref{catfied.fact.hlgy.of.enr.hom.functor} restricts to a map
\[
(\iC)^{\times R^{(0)}}
\longra
\cV^{\times R^{(1)}}
~.
\]
Of course, this is given by nothing other than the enriched hom functor \Cref{enr.hom.functor.for.C} of $\cC$.  

Finally, when $\cV$ is additionally \textit{symmetric} monoidal, there exists a \textit{total tensor product} functor
\[
\int_{|\D(M)} \BV
\xlongra{\bigboxtimes}
\cV
~.
\]
Using this, we define the \bit{enriched factorization homology} of $\cC$ over $M$ 
by the formula
\[
\int_M \cC
:=
\colim
\left(
\int_{|\D(M)} \cdiC
\xra{\int_{|\D(M)} \ulhom_\cC}
\int_{|\D(M)} \BV
\xlongra{\bigboxtimes}
\cV
\right)
~.
\]

Note that this indeed precisely rigorizes the heuristic definition of formula \Cref{heuristic.formula.for.enr.fact.hlgy}.  First of all, this colimit is indexed over the data of a disk-refinement $D \in \D(M)$ and an object in
\[
\int_R \cdiC
~,
\]
i.e., a labeling $\lambda$ of its vertices by points of $\iC$, and on such an object the value of the diagram is precisely the monoidal product
\[
\underset{e \in R^{(1)}}{\standaloneboxtimes}
\ulhom_\cC ( \lambda(s(e)) , \lambda(t(e)) )
\]
in $\cV$.
Moreover, its structure maps are given either by disappearances of points, which are taken to composition maps (as map \Cref{composition.in.enr.cat}), or by anticollisions of points, which are taken to unit maps (as map \Cref{unit.in.enr.cat}); the isotopies of configurations of points are encoded implicitly as equivalences in $\D(M)$.

This definition of enriched factorization homology has the key feature of being optimized for the isolation of the various specific properties of the enriching $\infty$-category $\cV$: all natural operations thereon have only to do with the factorization homology of $\BV$, while the factorization homology of $\cdiC$ simply comes along for the ride.  
This paradigm is implemented in the subsequent work~\cite{AMR-cyclo} to obtain the cyclotomic structure on $\THH$.  


\subsection{Miscellaneous remarks}

\begin{remark}
The work \cite{circle} clarifies the manifold-theoretic origins of a number of categories of lasting interest: the paracyclic category $\para$, the cyclic category $\cyclic$, and the epicyclic category $\epicyc$ (see \Cref{obs.para.cyclic.and.epicyclic.project.as.cocart.fibns}).  Though these admit combinatorial descriptions (through which they were originally defined), we find them easiest to manipulate when incarnated through manifolds: then, the geometry keeps track of the combinatorics.
\end{remark}


\begin{remark}
There are multiple notions of ``enriched $(\infty,n)$-categories''; among these, we expect that TQFTs should be most directly connected to those that are ``only enriched in dimension $n$''.  This is the notion we employ here, in the case that $n=1$.
\end{remark}

\begin{remark}
A key technical result in this paper (\Cref{all.about.extended.cart.s.m.deloop}) furnishes an extension of the $\infty$-operad $\BSVx \da \Fin_*$ corresponding to a cartesian symmetric monoidal $\infty$-category $(\cV,\times)$: we produce a cocartesian fibration $\wBSVx \da \Corr(\Fin)$ over the $\infty$-category of correspondences of finite sets, which sits in a canonical pullback square
\[ \begin{tikzcd}
\BSVx
\arrow{r}
\arrow{d}
&
\wBSVx
\arrow{d}
\\
\Fin_*
\arrow{r}
&
\Corr(\Fin)
\end{tikzcd} \]
and admits a \textit{total cartesian product} functor
\[
\wBSVx
\xlongra{\prod}
\cV
~.
\]
These data encode diagonal maps in $\cV$, as well as their interaction with products.  We find this result to be of independent interest, and we expect that it will be useful elsewhere.
\end{remark}

\subsection{Outline}

This paper is organized as follows.
\begin{itemize}

\item In \Cref{section.strat.mflds} we recall relevant foundational material from \cite{AFR-fact}, which is further developed in \cite{circle}, concerning an $\infty$-category $\M$ of stratified 1-manifolds and factorization homology over such.

\item Then, in \Cref{section.categorified.fact.hlgy} we review the theory of flagged enriched $(\infty,1)$-categories (as defined in \cite{GH-enr}) and study their categorified factorization homology.

\item Next, in \Cref{section.enriched.fact.hlgy} we define enriched factorization homology over an arbitrary stratified stratified space, and endow it with its functoriality for automorphisms of stratified stratified spaces.

\item Thereafter, in \Cref{section.cartesian.enr.fact.hlgy} we specialize to the case that our enriching $\infty$-category is \textit{cartesian} symmetric monoidal, and deduce vast additional functoriality of enriched factorization homology via the resulting diagonal maps.

\item In \Cref{section.fact.syst} we prove a result about factorization systems.
\end{itemize}

\subsection{Notation and conventions}
\label{subsection.notation.and.conventions}

\begin{enumerate}

\item \catconventions \inftytwoconventions

\item \functorconventions

\item \circconventions

\item \kanextnconventions

\item \spacescatsspectraconventions

\item \subcatconventions

\item \label{fibrationconventions} \fibrationconventions \dualfibnsconvention

\item \efibconventions

\item \cavalieraboutbicompleteness


\end{enumerate}

\stopcontents[sections]

\section{Recollections of factorization homology in dimension 1}
\label{section.strat.mflds}

In sections~\S\ref{sec.D} and~\S\ref{sec.M} we discuss the $\infty$-categories $\M$ of \bit{solidly 1-framed stratified spaces} (originally introduced as \cite[Definition 2.48]{AFR-fact}), and $\D$ of \bit{solidly 1-framed disk-stratified spaces} (originally introduced as \cite[Definition 3.26]{AFR-fact}), as characterized in~\cite{circle}.  
In~\S\ref{subsection.catfied.fact.hlgy.for.real} we define factorization homology over objects in $\M$ of category-objects internal to an $\infty$-category $\cX$ following \cite{circle}.
We only provide a quick summary suited for our purposes, and refer a reader to those works for details and context.

\subsection{The category $\D$}
\label{sec.D}

Consider the full subcategory $\bDelta_{\leq 1} \subset \bDelta$ consists of those finite non-empty linearly ordered sets with cardinality at most $2$.
Recall from \cite{circle} that we have a monomorphism
\begin{equation}
\label{e1}
{\sf Free}
\colon
{\sf di.Graphs^{\sf fin}}
\to
\Cat_{(\infty,1)}
~,
\end{equation}
from the category of finite directed graphs, which associates to a finite directed graph 
$D$ the free $(\infty,1)$-category on $D$.
The $\infty$-category $\Quiv$ of \bit{(finite) quivers} is  the full subcategory
\[
\Quiv
~\subset~
\Cat_{(\infty,1)}
\]
consisting of those 1-categories that are free on finite directed graphs.

\begin{notation}
We denote the following distinguished objects of $\Quiv$.
\begin{itemize}
\item $\DD^0$ is the 0-disk (i.e., the directed graph with one vertex and no edges),
\item $\DD^1$ is the closed 1-disk (i.e., the connected directed graph with two vertices and one edge),
\item $\SS^1_*$ is the pointed circle (i.e., the directed graph with one vertex and one edge).
\end{itemize}
\end{notation}


\begin{observation}[\cite{circle}]\label{t6}
The \bit{simplex} category and the \bit{epicyclic} category are the $\infty$-subcategories
\[
\bDelta
~\subset~
\Cat_{(\infty,1)}
\qquad\text{ and }\qquad
\w{\bLambda}
~\subset~
\Cat_{(\infty,1)}
\]
respectively consisting of those categories that are the values of ${\sf Free}$ on finite directed graphs that are linearly directed and that are cyclically directed, and respectively consisting of all functors between them and the non-constant functors between them.
In particular, there are fully faithful inclusions
\[
\bDelta
~\hookrightarrow~
\Quiv
\qquad\text{ and }\qquad
\w{\bLambda}
~\hookrightarrow~
\Quiv
~.
\]

\end{observation}

Observation~\ref{t6} offers the following.
\begin{definition}
\label{notation.brackets.functor.from.delta.op}
The \bit{cellular realization} functor is the fully faithful functor
\[
\rho
\colon
\bDelta^{\op}
\hookrightarrow
\Quiv
~.
\]
The value of $\rho$ on an object $[p]^\circ \in \bDelta$ is simply denoted $\rho(p)\in \Quiv$.

\end{definition}

\begin{lemma}
[\cite{circle}]
\label{t20}
The $\infty$-category $\Quiv$ and its classes of morphisms have the following features.
\begin{enumerate}
\item
$\Quiv$ is an ordinary category.




\item
$\Quiv$ admits finite coproducts, which are given by disjoint unions of finite directed graphs.
Furthermore, the fully faithful inclusion $\Quiv \hookrightarrow \Cat_{(\infty,1)}$ preserves finite coproducts,


\item
The restricted Yoneda functor
\[
\fCat_{(\infty,1)}
\longrightarrow
\PShv(\Quiv)
~,\qquad
\cC
\longmapsto
\left(
D
\longmapsto
\Hom_{\Cat_{(\infty,1)}}(D , \cC)
\right)
~,
\]
is fully faithful, whose image consists of presheaves that carry (the opposites of) a certain class of diagrams in $\Quiv$ to limit diagrams.

\end{enumerate}

\end{lemma}

The values of the fully faithful functor of Lemma~\ref{t20}(3) are rather simple, as in the following.
Indeed, as described in~\cite[Observation 1.1.4]{circle} articulates that $\hom_{\fCat_{(\infty,1)}}\bigl(D, \cC \bigr)$ is the moduli space of \bit{$\cC$-labels of $D$}, a point in which is an object in $\cC$ for each vertex in $D$, a morphism in $\cC$ for each (non-degenerate) edge in $D$, together with source-target compatibilities.  
%
%
%
%

\medskip

We next scrutinize the inclusion $\w{\bLambda} \hookrightarrow \Quiv$ of Observation~\ref{t6}.
%
%
%

\begin{definition}[{\cite[Definition 2.1.5]{circle}}]
\label{d4}
The \bit{Witt monoid} is the semi-direct product monoid-object in $\Spaces$:
\[
\WW 
~:=~
\TT \rtimes \NN^\times
~,
\]
where $\TT$ is the circle-group, $\NN^\times$ is the multiplicative monoid of natural numbers, and the action $\NN^\times \lacts \TT$ is given by $n\cdot z := z^n$.  

\end{definition}

%

\begin{prop}
[\cite{circle}]
\label{t10}
There is a natural functor
\[
\w{\bLambda}
\longrightarrow
\fB \WW
\]
that is a cartesian fibration.
Furthermore, there are canonical base-changes among cartesian fibrations
\begin{equation}
\label{paracyclic.cyclic.epicyclic}
\begin{tikzcd}
\copara
\arrow[two heads]{r}
\arrow{d}{\sf loc}
&
\bLambda
\arrow[hook, two heads]{r}
\arrow{d}{\sf loc}
&
\w{\bLambda}
\arrow{d}{\sf loc}
\\
\ast
\arrow[two heads]{r}
&
\BT
\arrow[hook, two heads]{r}
&
\BW
\end{tikzcd}
~,
\end{equation}
in which the vertical functors are localizations, the horizontal functors are surjective, and the right horizontal functors are monomorphisms, 
and where $\copara$ the \bit{paracyclic category} (\cite{GetzlerJones-paracyclic}) and $\bLambda$ is the \bit{cyclic category} (\cite{Connes-coh}).  
In particular, the Witt monoid $\WW$ canonically acts on the paracyclic category $\copara$, and there are canonical identifications of the resulting $\TT$-coinvariants and the right-lax $\WW$-coinvariants:
\[
\left( \copara \right)_{\htpy \TT}
~\simeq~
\bLambda
\qquad\text{ and }\qquad
\left( \copara \right)_{\htpy^\rlax \WW}
~\simeq~
\w{\bLambda}
~.
\]

\end{prop}

\cite[Definition 3.26]{AFR-fact} introduces the $\infty$-category ${\sf c}\bcD{\sf isk}^{\sf sfr}_1$ of solidly 1-framed disk-stratified spaces.  
In this paper, we use the following shorthand.
\begin{notation}
$
\D
~:=~
{\sf c}\bcD{\sf isk}^{\sf sfr}_1
~.
$
\end{notation}
\cite[Lemma 3.44]{AFR-fact} constructs a fully faithful functor
\[
\bDelta^{\op}
\xra{~\langle - \rangle~}
\D
~.
\]
In this way, we regard $\D$ as an $\infty$-category under $\bDelta^{\op}$.
Similarly, via $\bDelta \xra{\rho} \Quiv$ of \Cref{notation.brackets.functor.from.delta.op}, we regard $\Quiv^\op$ as an $\infty$-category under $\bDelta^{\op}$.
A main result of \cite{circle} states the following.
\begin{theorem}[\cite{circle}]
\label{circle.main}
There is an equivalence between $\infty$-categories under $\bDelta^{\op}$:
\[
\D
~\simeq~
\Quiv^{\op}
~.
\]

\end{theorem}

\subsection{The category $\M$}
\label{sec.M}


Recall from \cite{AFR-fact} $\infty$-category of \bit{solidly 1-framed stratified spaces}, which, for simplicity, we denote as
\[
\M
~:=~
{\sf c}\bcM{\sf fd}^{\sf sfr}_1
~.
\]
By definition, there is a fully faithful inclusion
\[
\delta
\colon
\D
~\hookrightarrow~
\M
~.
\]
(See \cite{circle} which gives a universal property of $\M$ and which describes $\M$ combinatorially.)

In \cite[\S 6.6]{AFR-strat}, as well as \cite[Definition 1.14]{AFR-fact}, the define and study the following classes of morphisms in $\M$.  Here, we informally recall these classes of morphisms (see those citations for precise details).  
\begin{definition}
\label{d13}
Let $M \xra{f} N$ be a morphism in $\M$.
\begin{itemize}
\item
Say $f$ is \bit{closed-creation} if it corresponds to a proper constructible bundle $M \la N$. 

\item
Say $f$ is \bit{closed} if it corresponds to an inclusion $M \hookleftarrow N$ of a closed union of strata.

\item
Say $f$ is \bit{creation} if it corresponds to a surjective proper constructible bundle $M \la N$. 

\item
Say $f$ is \bit{refinement} if it corresponds to a refinement map between stratified spaces $M \to N$.  (Informally, a refinement map is a homeomorphism between underlying topological spaces, and $M$ has a finer stratification than $N$.)

\item
Say $f$ is \bit{active} if it is a composite of creation and refinement morphisms.

\end{itemize}

\end{definition}

\begin{notation}
We use the symbols $\cls$, $\act$, $\cre$, $\refi$, and $\cls.\cre$ to respectively refer to the subcategories of $\M$ consisting of closed, active, creation, refinement, and closed-creation morphisms.
\end{notation}

\begin{observation}
\label{obs.fact.syst.on.M}
By \cite[Theorem 6.5.6]{AFR-strat}, the pair $[\M^{\cls} ; \M^{\act}]$ defines a factorization system on $\M$: that is, every morphism in $\M$ admits a unique factorization as the composite
\[
\bullet
\stackrel{\cls}{\longra}
\bullet
\stackrel{\act}{\longra}
\bullet
\]
of a \textit{closed} morphism followed by an \textit{active} morphism.  On the other hand, any active morphism in $\M$ can be non-uniquely factored as the composite
\[
\bullet
\stackrel{\cre}{\longra}
\bullet
\stackrel{\refi}{\longra}
\bullet
\]
of a creation morphism followed by a refinement morphism \cite[\S 6.2]{AFR-strat}.
\end{observation}

As we will see in Definition~\ref{definition.cartesian.fact.hlgy}, factorization homology over $M$ is defined as a colimit indexed by the $\infty$-category $\D_{/M}$.  
We are therefore interested in final functors to $\D_{/M}$ from simpler $\infty$-categories.  
Proposition~\ref{t18}(2) below supplies such. 
To state it requires some set-up.
\begin{notation}
We write
\[
\totalD
:=
\lim
\left(
\begin{tikzcd}
&
\Ar^\refi(\M)
\arrow{d}{s}
\\
\D
\arrow[hook]{r}
&
\M
\end{tikzcd}
\right)
\]
for the full $\infty$-subcategory of the arrow $\infty$-category $\Ar(\M)$ consisting of those refinement morphisms whose source is disk-stratified.
\end{notation}

\begin{definition}
Let $M \in \M$ be a solidly 1-framed stratified stratified space.  
The category of \bit{disk-refinements} of $M$ is the pullback:
\[
\D(M)
:=
\lim
\left(
\begin{tikzcd}
&
\totalD \arrow{d}{t}
\\
\{ M \}
\arrow{r}
&
\M
\end{tikzcd}
\right)
\simeq
\lim
\left(
\begin{tikzcd}
&
&
\D
\\
&
\Ar^{\sf ref}(\M)
\arrow{r}{s}
\arrow{d}[swap]{t}
&
\M
\arrow[hookleftarrow]{u}
\\
\{ M \}
\arrow{r}
&
\M
\end{tikzcd}
\right)~.\footnote{A priori, $\D(M)$ is an $\infty$-category, but in fact it is always a 1-category (see~\cite{circle}).}
\]
\end{definition}

\begin{prop}[\cite{circle}]
\label{obs.para.cyclic.and.epicyclic.project.as.cocart.fibns}
The pullback diagrams of Observation~\ref{t10} extend as a pullback diagram among $\infty$-categories:
\begin{equation}
\label{e9}
\begin{tikzcd}
\para
\arrow[two heads]{r}
\arrow{d}{\sf loc}
&
\cyclic
\arrow[hook, two heads]{r}
\arrow{d}{\sf loc}
&
\epicyc
\arrow[hook]{r}{\ff}
\arrow{d}{\sf loc}
&
\totalD
\arrow{d}{\ev_t}
\\
\ast
\arrow[two heads]{r}{\langle \SS^1 \rangle}
&
\BT
\arrow[hook, two heads]{r}
&
\BW^{\op}
\arrow[hook]{r}[swap]{\ff}
&
\M
\end{tikzcd}
~.
\end{equation}
In particular, there is a canonical $\WW^{\op}$-equivariant equivalence 
\[
\bcD(\SS^1)
~\simeq~
\para
~.
\]

\end{prop}


The assignment $M\mapsto \D(M)$ is functorial in the following sense.

\begin{observation}
\label{pullback.of.disk.stratns.along.stratified.covering.spaces}
Write
\[
\M^{\sf cov}
\subset
\M^{\cls.\cre}
\]
for the subcategory on those closed-creation morphisms which are opposite to \textit{stratified covering spaces} (i.e., proper constructible bundles that are locally trivial (not just after restricting to strata)).  Then, the restriction
\[ \begin{tikzcd}
\left( \totalD \right)_{|\M^{\sf cov}}
\arrow{r}
\arrow{d}
&
\totalD
\arrow{d}{t}
\\
\M^{\sf cov}
\arrow[hook]{r}
&
\M
\end{tikzcd} \]
is a cocartesian fibration, with cocartesian pushforward functors given by pullback of disk-refinements.
\end{observation}


\begin{observation}
\label{obs.possible.maps.of.refinements}
The composite functor factors through the active subcategory:
\[
\xymatrix{
&&
&&
\D^{\sf act}
\ar[d]
\\
\D(M)
\ar[rr]
\ar@{-->}[urrrr]^-{\exists !}
&&
\totalD
\ar[rr]^-{s}
&&
\D
.
}
\]
Furthermore, for 
\[ 
\begin{tikzcd}
D_0
\arrow{rr}
\arrow{rd}[swap, sloped, pos=0.7]{\refi}
&
&
D_1
\arrow{ld}[sloped, pos=0.7]{\refi}
\\
&
M
&
\end{tikzcd}
\]
a morphism in $\D(M)$ that is carried to a creation-morphism $D_0 \xra{f} D_1$ in $\D^{\sf act}$,
then the restriction of the proper constructible bundle over the 1-dimensional strata,
\[
f_|
\colon
(D_1)_{|D_0^{(1)}} 
\longrightarrow
D_0^{(1)}
\]
is a homeomorphism.

\end{observation}


\subsection{Factorization homology}
\label{subsection.catfied.fact.hlgy.for.real}

Here, we recall factorization homology of category-objects internal to an $\infty$-category $\cX$, as introduced in~\cite{AFR-fact} in the case that $\cX = \Spaces$, and modestly developed further in~\cite{circle} for general $\cX$.

\begin{notation}
\label{fact.system.on.bDeltaop}
Through the fully faithful embedding
\[
\bDelta^{\op}
\stackrel{\rho}{\longhookra}
\D
~,
\]
the category $\bDelta^{\op}$ inherits a factorization system, which we denote by
\[ [\bDelta^{\op,\cls};\bDelta^{\op,\act}]~.\]
\end{notation}

\begin{remark}
The factorization system
\[ [\bDelta^{\op,\cls};\bDelta^{\op,\act}] \]
on $\bDelta^\op$ is opposite to a factorization system
\[ [\bDelta^\act ; \bDelta^\cls] \]
on $\bDelta$, which can be described explicitly as follows:
\begin{itemize}
\item the \bit{active} morphisms in $\bDelta$ are those that preserve minimum and maximum elements, and
\item the \bit{closed} morphisms in $\bDelta$ are the consecutive inclusions, i.e., those morphisms of the form $i \mapsto i+k$ for some fixed $k$.
\end{itemize}
\end{remark}

\begin{remark}
The morphisms in $\bDelta^{\op}$ that we call ``closed'' are elsewhere called ``inert'' (e.g.\! in \cite{LurieHA}).  We choose our terminology to remain consistent with that surrounding the $\infty$-category $\M$.
\end{remark}

For the next definition, recall the subcategory
\[
\bDelta^{\sf cls}
~\subset~
\bDelta
\]
consisting of the same objects yet only those morphisms that are \bit{closed}, which is to say convex inclusions.  
See \cite[Section 1.5]{circle} for contextualizing the following.
\begin{definition}
\label{d11}
Let $\cX$ be an $\infty$-category.
The $\infty$-category of \bit{category-objects internal to $\cX$} is the full $\infty$-subcategory
\[
\fCat_1[\cX]
~\subset~
\Fun(\bDelta^{\op},\cX)
\]
consisting of those functors
\[
\cC
\colon
\bDelta^{\op}
\longrightarrow
\cX
\]
for which the composite functor $(\bDelta^{\sf cls})^{\op} \hookrightarrow \bDelta^{\op} \xra{\cC} \cX$ carries (the opposites of) colimit-diagrams to limit-diagrams.

\end{definition}

%

The following definition is \cite[Definition 4.13]{AFR-fact} in the case that $n=1$ and $\cX=\Spaces$, and in~\cite{circle} for the general case.
\begin{definition}
\label{definition.cartesian.fact.hlgy}
Let $\cX$ be an $(\infty,1)$-category.
For $\cC \in \Fun(\bDelta^\op,\cX)$ and $M \in \M$, the \bit{factorization homology} of $\cC$ over $M$ is the colimit in $\cX$:
\[
\int_M \cC
~=~
\colim\Bigl(
\D_{/M}
\xra{\rm forget}
\D
\xra{~D\mapsto \lim \bigl( (\bDelta^{\op})^{D/} \xra{\rm forget} \bDelta^{\op} \xra{\cC} \cX \bigr)~}
\cX
\Bigr)
~
\in \cX
~.
\]
The \bit{factorization homology} functor is right Kan extension along $\rho$ followed by left Kan extension along $\delta$:
\[
\int_{(=)}(-) :
\Fun(\bDelta^\op,\cX)
\xra{~\rho_*~}
\Fun(\D,\cX)
\xra{~\inclDM_!~}
\Fun(\M,\cX)
~.
\]
\end{definition}


The following summarizes some key properties of factorization homology established in~\cite{circle}.
\begin{prop}[\cite{circle}]\label{t18}
Let $\cX$ be an $\infty$-category.
Let $\cC \in \fCat_1[\cX]\subset \Fun(\bDelta^{\op},\cX)$ be a category-object internal to $\cX$.
\begin{enumerate}
\item
Let $D\in \D$.
There is a canonical functor from the exit-path $\infty$-category to the $\infty$-undercategory:
\[
{\sf Exit}(D)^{\op}
\xra{~\rm dimension~} 
(\bDelta_{\leq 1}^{\op})^{D/}
~\subset~
(\bDelta^{\op})^{D/}
~,\qquad
(x\in D)
\mapsto 
[ {\sf dim}_x(D) ]
~,
\]
where ${\sf dim}_x(D)$ is the dimension of the stratum of $D$ in which $x\in D$ belongs.
This functor has the property that the resulting canonical morphism in $\cX$,
\[
\rho_* \cC(D)~=~
\lim \bigl( (\bDelta^{\op})^{D/} \xra{\rm forget} \bDelta^{\op} \xra{\cC} \cX \bigr)
\longrightarrow
\lim \bigl( \Exit(D)^{\op} \xra{\rm dimension} (\bDelta^{D/})^{\op} \xra{\rm forget} \bDelta^{\op} \xra{\cC} \cX \bigr)
~,
\]
is an equivalence.
In particular, because the $\infty$-category $\Exit(D)$ is finite, the value $\rho_\ast \cC(D)$ is a finite limit in $\cX$ of $\Obj(\cC)$ and $\Mor(\cC)$.

\item
Let $M\in \M$.
The canonical functor
\[
\D(M)
\longrightarrow
\D_{/M}
\]
is final.
In particular, the canonical morphism in $\cX$,
\[
\int_M \cC
\xla{~\simeq~}
\colim\Bigl(
\D(M)
\to \D_{/M}
\xra{\rm forget}
\D
\xra{~D\mapsto \lim \bigl( (\bDelta^{\op})^{D/} \xra{\rm forget} \bDelta^{\op} \xra{\cC} \cX \bigr)~}
\cX
\Bigr)
~,
\]
is an equivalence.
Furthermore, the category $\D(M)$ receives a final functor from $\bDelta^{\op}$.

\end{enumerate}

\end{prop}

\begin{remark}
\cite[Corollary 4.1.6]{circle} establishes the following fact.
For $\cC \in \fCat_1[\cX]\subset \Fun(\bDelta^{\op} , \cX)$ a category-object internal to $\cX$, and for $M\in \M$, the factorization homology $\int_M \cC$ exists provided $\cX$ admits finite limits and geometric realizations and products distribute over geometric realizations.
\end{remark}

%


\begin{remark}
As explained in \Cref{remark.cartesian.enr.fact.hlgy.recovers.AFR}, factorization homology agrees with cartesian enriched factorization homology where the two are comparable.  
\end{remark}

\begin{example}[\cite{circle}]
\label{r7}
Let $\cX$ be an $\infty$-category with finite limits and geometric realizations.
Let $\cC\in \fCat_1[\cX]$ be a category-object internal to $\cX$.
Its factorization homology takes the following values.
\begin{enumerate}
\item
\[
\int_{\DD^0} \cC = \cC([0]) =: \Obj(\cC)
~\in \cX
~,\qquad
\text{ and }
\qquad
\int_{\DD^1} \cC = \cC([1]) =: \Mor(\cC)
~\in \cX
~,
\]
\[
\text{ and more generally }
\qquad
\int_{\rho(p)} \cC = \cC([p]) 
~\in \cX
\qquad
\text{for each }[p]\in \bDelta
~.
\]

\item
For $D\in \D$,
\[
\int_{D} \cC \simeq \lim\Bigl(
\Exit(D)^{\op} \xra{\rm dimension} \bDelta_{\leq 1}^{\op}
\hookrightarrow
\bDelta^{\op}
\xra{\cC} 
\cX
\Bigr)
~,
\]
where $\Exit(D)^{\op} \xra{\rm dimension} \bDelta_{\leq 1}^{\op}$ evaluates on an object in an $i$-dimensional stratum as $[i]$, and on an exiting-path that agrees/disagrees with the direction of an edge to $[0] \xra{ \langle 0 \rangle / \langle 1 \rangle } [1]$.

\item\label{example:circle}
The value of $\int \cC$ on an oriented circle is the geometric realization of the cyclic bar construxtion of $\cC$:
\[
\int_{\SS^1} \cC
\xra{~\simeq~}
\bigl|
{\sf Bar}_\bullet^{\sf cyc}(\cC)
\bigr|
=: {\sf HH}(\cC)
~
\in \cX
~.
\]

\item
For $M ,M'\in \M$, 
the canonical morphism is an equivalence:
\[
\int_{M\sqcup M'} \cC
\xra{~\simeq~}
\int_M \cC
\times
\int_{M'} \cC
~
\in
\cX
~.
\]

\item
Because every object in $\M$ is a finite disjoint union of oriented circles and a finite directed graph, the previous three points determine the value of $\int \cC$ on any object in $\M$.

\end{enumerate}

\end{example}

\section{Right-lax functoriality of factorization homology}
\label{section.categorified.fact.hlgy}

In this section, we use the unstraightening construction to recast factorization homology of category-objects in $\Cat$ in terms of coCartesian fibrations: we call this \bit{categorified factorization homology}.
Once so arranged, we notice as  \Cref{subsection.catfied.fact.hlgy.of.left.lax.functors} that such factorization homology possesses a more functoriality than expected.
We use this gained functoriality to define enriched factorization homology in \Cref{section.enriched.fact.hlgy}.  
We begin with \Cref{subsection.flagged.enr.cats}, which is devoted to the definition of flagged enriched $(\infty,1)$-categories.

\begin{definition}
\label{d.cat.fact}
\bit{Categorified factorization homology} is the composition of factorization homology for category-objects internal to $\Cat$ followed by the unstraightening construction:
\[
\fCat_1[\Cat]
\xra{\int}
\Fun( \M  , \Cat )
~\simeq~
\coCart_{\M}
~,\qquad
\cC
\mapsto
(
\int_{|\M} \cC
\to
\M
)
~.
\]

\end{definition}

\begin{notation}
\label{d.cat.fact.over}
Let $\cC$ be a category-object in $\Cat$.
For $\cK \to \M$ a functor, we denote the pullback $\infty$-category:
\[
\xymatrix{
\int_{|\cK} \cC
\ar[rr]
\ar[d]
&&
\int_{|\M} \cC
\ar[d]
\\
\cK
\ar[rr]
&&
\M
.
}
\]

\end{notation}

\subsection{Flagged enriched $(\infty,1)$-categories}
\label{subsection.flagged.enr.cats}

\begin{notation}
Given a monoidal $\infty$-category $(\cV,\boxtimes)$, we write
\[ \begin{tikzcd}
\BVbox
\arrow{d}
\\
\bDelta^\op
\end{tikzcd} \]
for its \bit{monoidal deloop}, i.e., its corresponding nonsymmetric $\infty$-operad.\footnote{
One way to construct $\BVbox$ is as the unstraightening of the bar-construction 
${\sf Bar}_\bullet \bDelta^{\op}
\to 
\Cat_1$.
}
When the monoidal structure is understood, we simply write
\[
\BV
:=
\BVbox
~.
\]
\end{notation}

The underlying $\infty$-groupoid of a flagged $\cV$-enriched $(\infty,1)$-category participates in its definition via the following.

\begin{definition}
\label{define.infinite.join}
The \bit{codiscrete category-object} functor is the right Kan extension
\[ \begin{tikzcd}[column sep=2cm]
\Fun(\bDelta^\op,\Spaces)
\arrow[transform canvas={yshift=0.9ex}]{r}{(-)_{|[0]^\circ}}
\arrow[dashed, hookleftarrow, transform canvas={yshift=-0.9ex}]{r}[yshift=-0.2ex]{\bot}[swap]{\cd}
&
\Spaces
\end{tikzcd} \]
along the fully faithful inclusion of the initial object $[0]^\circ \in \bDelta^\op$.
\end{definition}

\begin{observation}
The functor $\cd$ takes values in the full subcategory of $\Fun(\bDelta^\op,\Spaces)$ consisting of the category-objects in $\Spaces$ (as its name suggests).  
Hence, it does too when considered as landing in $\Fun(\bDelta^\op,\Cat)$ since the inclusion $\Spaces \subset \Cat$ preserves fiber products.
\end{observation}

\begin{remark}
\label{remark.segal.spaces}
Category-objects in $\Spaces$ (i.e., Segal spaces) are equivalent to \textit{flagged $(\infty,1)$-categories} in the sense of \cite{AF-flagged}.  This is to say that a Segal space is equivalently specified by an $(\infty,1)$-category, called its \bit{underlying $(\infty,1)$-category}, equipped with a surjective functor from an $\infty$-groupoid.  
Thereafter, a Segal space is \textit{complete} just when this functor is the inclusion of its maximal subgroupoid.  
For a space $V \in \Spaces$, the Segal space $\cd(V)$ corresponds to the flagged $(\infty,1)$-category
\[
V
\longra
\tau_{\leq (-1)} V
~,
\]
the canonical map from $V$ to its $(-1)$-truncation.  
In other words, the underlying $(\infty,1)$-category of $\cd(V)$ is $\pt$ when $V$ is nonempty and $\es$ when $V$ is empty.
\end{remark}

We now proceed to define right-lax functors between category-objects in $\Cat$.  This requires a some preliminaries.
Recall the Notation~\ref{fact.system.on.bDeltaop} for $\bDelta^{\sf cls}$ and $\bDelta^{\sf act}$.
\begin{notation}
\label{notation.cls.cocart}
Given a pair $\cB \supset \cB_0$ of an $\infty$-category and a subcategory thereof, we write
\[
\begin{tikzcd}[row sep=1.5cm]
\Cat^{\cB_0}_{\cocart/\cB}
\arrow[hook, two heads]{r}
\arrow{d}
&
\Cat_{\cocart/\cB}
\arrow{d}
\\
\coCart_{\cB_0}
\arrow[hook, two heads]{r}
&
\Cat_{\cocart/\cB_0}
\end{tikzcd}
\]
for the pullback, where the right vertical functor is given by pullback along the inclusion and the bottom functor is the canonical surjective monomorphism.  In the special case of the pair $\bDelta^\op \supset \bDelta^{\op,\cls}$ we simply write
\[
\Cat^{\cls}_{\cocart/\bDelta^\op}
:=
\Cat^{\bDelta^{\op,\cls}}_{\cocart/\bDelta^\op}
~.
\]
\end{notation}

\begin{remark}
The $\infty$-category $\Cat^{\cB_0}_{\cocart/\cB}$ can be informally described as follows:
\begin{itemize}
\item its objects are the cocartesian fibrations over $\cB$, and
\item its morphisms are those functors over $\cB$ that preserve cocartesian lifts of morphisms in $\cB_0 \subset \cB$.
\end{itemize}
\end{remark}

\begin{definition}
Let $\sC, \sD \in \fCat_1[\Cat_1]\subset \Fun(\bDelta^{\op},\Cat_1)$ be category-objects in $\Cat_1$.
A \bit{right-lax functor} from $\sC$ to $\sD$ is a morphism in $\Cat^\cls_{\cocart/\bDelta^\op}$ 
\[
\xymatrix{
\cC
\ar[dr]
\ar@{-->}[rr]
&&
\cD
\ar[dl]
\\
&
\bDelta^{\op}
&
}
\]
between the cocartesian fibrations that they classify.
A \bit{(strict) functor} from $\sC$ to $\sD$ is a right-lax functor that belongs to the $\infty$-subcategory $\coCart_{\bDelta^{\op}}\subset \Cat^\cls_{\cocart/\bDelta^\op}$.
\end{definition}

\begin{remark}
In light of the unstraightening equivalence between $\infty$-categories, ${\sf Un}\colon \Fun(\bDelta^{\op} , \Cat_1) \simeq \coCart_{\bDelta^{\op}}$, a strict functor from $\sC$ to $\sD$ corresponds to a natural transformation from $\sC$ to $\sD$ as simplicial $(\infty,1)$-categories.

\end{remark}

In particular, we have the following.

\begin{definition}[after~\cite{GH-enr}]
A \bit{flagged $\cV$-enriched $(\infty,1)$-category} $\cC$ is the data of an $\infty$-groupoid $\iC \in \Spaces$, called its \bit{underlying $\infty$-groupoid}, equipped with a right-lax functor
\[
\begin{tikzcd}
\cdiC
\arrow{rr}{\ulhom_\cC}
\arrow{rd}
&
&
\BV
\arrow{ld}
\\
&
\bDelta^{\op}
\end{tikzcd} \]
between category-objects in $\Cat_1$.  
These form an evident $\infty$-category, which we denote by $\fCat_1(\cV)$.
\end{definition}

\subsection{Right-lax functoriality of factorization homology}
\label{subsection.catfied.fact.hlgy.of.left.lax.functors}

Fix a flagged $\cV$-enriched $(\infty,1)$-category $\cC$.  In order to define its enriched factorization homology, we will want to construct a morphism
\[ \begin{tikzcd}[column sep=2cm]
\dint_{\! \! \! |\D} \cdiC
\arrow[dashed]{rr}{\int_{|\D} \ulhom_\cC}
\arrow{dr}
&
&
\dint_{\! \! \! |\D} \BV
\arrow{ld}
\\
&
\D
&
\end{tikzcd} \]
in $\Cat_{\cocart/\D}$.  However, this does not come for free: a priori, the right Kan extension
\[
\Fun(\bDelta^\op,\Cat_1)
\xra{\rho_*}
\Fun(\D,\Cat_1)
\]
of \Cref{definition.cartesian.fact.hlgy} is only functorial for \textit{strict} (and not right-lax) functors among category-objects.  To obtain it, we will exploit the closed-active factorization system on $\bDelta^\op$ via the following result.

\begin{prop}
\label{freely.add.active.cocart}
There exists a left adjoint
\begin{equation}
\label{adjn.free.act}
\begin{tikzcd}[column sep=2cm]
\Cat^{\cls}_{\cocart/\bDelta^{\op}}
\arrow[dashed, transform canvas={yshift=0.9ex}]{r}{\Freeact}
\arrow[\surjmonoleft, transform canvas={yshift=-0.9ex}]{r}[yshift=-0.2ex]{\bot}
&
\coCart_{\bDelta^{\op}}
\end{tikzcd}
\end{equation}
to the surjective monomorphism, which takes an object $(\cE \ra \bDelta^\op) \in \Cat^\cls_{\cocart/\bDelta^\op}$ to the horizontal composite in the diagram
\[
\begin{tikzcd}
\cE
\underset{\bDelta^\op}{\times}
\Ar^{\act}(\bDelta^\op)
\arrow{r}
\arrow{d}
&
\Ar^{\act}(\bDelta^\op)
\arrow{r}{\ev_t}
\arrow{d}{\ev_s}
&
\bDelta^\op
\\
\cE
\arrow{r}
&
\bDelta^\op
\end{tikzcd}~.
\]
\end{prop}

\begin{proof}
This is the special case of \Cref{freely.add.cocart.for.second.factor} in which we take $\cB = \bDelta^\op$ and $[\cB_0;\cB_1] = [\bDelta^{\op,\cls};\bDelta^{\op,\act}]$.
\end{proof}

\begin{remark}
As suggested by the notation, the left adjoint in \Cref{freely.add.active.cocart} can be informally described as freely adjoining cocartesian lifts for active morphisms in $\bDelta^\op$; this generalizes the ``free cocartesian fibration'' construction of \cite{GHN}.
\end{remark}

\begin{notation}
For a space $V \in \Spaces$, we just write
\[
\Freeact(V)
:=
\Freeact(\cd(V))
~,
\]
for simplicity.
\end{notation}

\begin{observation}
\label{obs.describe.Free.act}
The object $(\Freeact(V) \da \bDelta^\op) \in \coCart_{\bDelta^\op}$ can be heuristically described as follows: an object in the fiber over $[n]^\circ \in \bDelta^\op$ is given by the data of an active morphism $[m]^\circ \ra [n]^\circ$ and a labeling $[m] \ra V$ (of the elements of the poset $[m]$ by points in $V$).  In particular, it defines (i.e., straightens to) a category-object in $\Cat_1$,
\begin{itemize}
\item whose $\infty$-category of objects is $V \in \Spaces \subset \Cat_1$,
\item in which a morphism from $v_0 \in V$ to $v_1 \in V$ is given by the data of
\begin{itemize}
\item a natural number $m \geq 1$ (corresponding to the unique active morphism $[m]^\circ \ra [1]^\circ$) along with
\item a sequence
\[
v_0=w_0
~,~
w_1
~,~
w_2
~,~
\ldots
~,~
w_m = v_1
\]
of $m+1$ points in $V$, starting with $v_0$ and ending with $v_1$,
\end{itemize}
and
\item in which composition is given by concatenation of sequences.
\end{itemize}
\end{observation}

\begin{remark}
For a space $V \in \Spaces \subset \Cat_1$, the object $\Freeact(V) \in \coCart_{\bDelta^\op}$ is a ``generalized'' monoidal $\infty$-category (in the sense of the generalized $\infty$-operads of \cite[\S 2.3.2]{LurieHA}): it satisfies the Segal condition, but its fiber over $[0]^\circ \in \bDelta^\op$ isn't contractible.  Modulo this detail, the unit map
\[ \begin{tikzcd}
\cd(V)
\arrow{rr}
\arrow{rd}
&
&
\Freeact(V)
\arrow{ld}
\\
&
\bDelta^\op
\end{tikzcd} \]
may be thought of as ``the free enriched flagged $\infty$-category with underlying $\infty$-groupoid $V$'' (with functoriality for \textit{strict} monoidal functors of generalized monoidal $\infty$-categories).
\end{remark}

\begin{construction}
\label{construct.hom.dagger}
For any flagged $\cV$-enriched $(\infty,1)$-category $\cC$, by \Cref{freely.add.active.cocart} there exists a canonical factorization
\[
\begin{tikzcd}[column sep=1.5cm]
\cdiC
\arrow{r}
\arrow[bend left]{rr}{\ulhom_\cC}
\arrow{rd}
&
\Freeact(\iC)
\arrow[dashed]{r}{\ulhom_\cC^\dagger}
\arrow{d}
&
\BV
\arrow{ld}
\\
&
\bDelta^{\op}
\end{tikzcd}
\]
in which
\begin{itemize}
\item the first map is the unit of the adjunction \Cref{adjn.free.act}, and so lies in $\Cat^\cls_{\cocart/\bDelta^\op}$, and
\item the second map lies in $\coCart_{\bDelta^\op}$.
\end{itemize}
The latter map is taken by the right Kan extension
\[
\coCart_{\bDelta^\op}
\simeq
\Fun(\bDelta^\op,\Cat_1)
\xra{\rho_*}
\Fun(\D,\Cat_1)
\simeq
\coCart_\D
\]
 to a map
\begin{equation}
\label{int.D.of.hom.C.dagger}
\begin{tikzcd}[column sep=1.5cm]
\dint_{\! \! \! |\D} \Freeact(\iC)
\arrow{r}{\int_{|\D} \ulhom_\cC^\dagger}
\arrow{d}
&
\dint_{\! \! \! |\D} \BV
\arrow{ld}
\\
\D
\end{tikzcd}
\end{equation}
in $\coCart_\D$.
\end{construction}

\begin{observation}
\label{fact.hlgy.of.iCj.as.subcat}
There is a canonical monomorphism
\begin{equation}
\label{fact.hlgy.over.D.of.unit}
\begin{tikzcd}
\dint_{\! \! \! |\D} \cdiC
\arrow[hook]{r}
\arrow{d}
&
\dint_{\! \! \! |\D} \Freeact(\iC)
\arrow{ld}
\\
\D
\end{tikzcd}
\end{equation}
in $\Cat_{\cocart/\D}$, namely the identification of the source as the full subcategory of the target on those disk-stratified stratified spaces whose framed intervals are labeled by morphisms in $\Freeact(\iC)$ associated to the natural number 1 (corresponding to the unique active morphism $[1]^\circ \ra [1]^\circ$ (recall \Cref{obs.describe.Free.act})) -- that is, in which the sequence of points of $\iC$ has exactly two elements.
\end{observation}

\begin{construction}
\label{construct.fact.hlgy.over.D.of.hom.C}
From \Cref{construct.hom.dagger} \and \Cref{fact.hlgy.of.iCj.as.subcat}, we obtain a canonical composite
\begin{equation}
\label{composite.giving.categorified.fact.hlgy.over.D}
\begin{tikzcd}[column sep=1.5cm]
\dint_{\! \! \! |\D} \cdiC
\arrow[hook]{r}{\Cref{fact.hlgy.over.D.of.unit}}
\arrow[dashed, bend left]{rr}{\int_{|\D} \ulhom_\cC}
\arrow{d}
&
\dint_{\! \! \! |\D} \Freeact(\iC)
\arrow{r}{\int_{|\D} \ulhom_\cC^\dagger}
\arrow{ld}
&
\dint_{\! \! \! |\D} \BV
\arrow[bend left=15]{lld}
\\
\D
\end{tikzcd}
\end{equation}
in $\Cat_{\cocart/\D}$.
\end{construction}

\begin{remark}
The composite \Cref{composite.giving.categorified.fact.hlgy.over.D} takes an object
\[
\left(
D \in \D
~,~
D^{(0)} \xlongra{\lambda} \iC
\right)
\in
\int_{|\D} \cdiC
\]
to the object
\[
\left( D \in \D
~,~
D^{(1)} \longra \cV
\right)
\in
\int_{|\D} \BV
\]
determined by the prescription
\[ \begin{tikzcd}[row sep=0cm]
D^{(1)}
\arrow{r}
&
\cV
\\
\rotatebox{90}{$\in$}
&
\rotatebox{90}{$\in$}
\\
e
\arrow[mapsto]{r}
&
\ulhom_\cC(\lambda(s(e)),\lambda(t(e)))
\end{tikzcd}~. \]
The source is a left fibration over $\D$, so every morphism therein is cocartesian over $\D$.  In view of  \Cref{obs.fact.syst.on.M}, the action of the composite \Cref{composite.giving.categorified.fact.hlgy.over.D} is completely specified by the following facts.
\begin{itemize}

\item A closed morphism $D \ra D'$ in $\D$ determines a commutative diagram
\[ \begin{tikzcd}
D^{(0)}
\arrow[hookleftarrow]{r}
&
D'^{(0)}
\\
D^{(1)}
\arrow[hookleftarrow]{r}
\arrow{u}{s}
\arrow{d}[swap]{t}
&
D'^{(1)}
\arrow{u}[swap]{s}
\arrow{d}{t}
\\
D^{(0)}
\arrow[hookleftarrow]{r}
&
D'^{(0)}
\end{tikzcd} \]
of finite sets.  A (necessarily cocartesian) morphism in the source $\int_{|\D} \cdiC$ over this morphism in $\D$ must be determined by restriction along the map $D'^{(0)} \ra D^{(0)}$: the labeling of the vertices of $D'$ must be pulled back from those of $D$ via the defining proper constructible bundle.  The composite \Cref{composite.giving.categorified.fact.hlgy.over.D} then takes this to another cocartesian morphism in the target $\int_{|\D} \BV$, likewise given by restriction along the monomorphism $D'^{(1)} \ra D^{(1)}$.

\item A creation morphism $D \ra D'$ in $\D$ determines a commutative diagram
\[ \begin{tikzcd}
D^{(0)}
&
&
D'^{(0)}
\arrow[two heads]{ll}
\\
D^{(1)}
\arrow{u}{s}
\arrow{d}[swap]{t}
&
\left( D'^{(1)} \right)_{|D^{(1)}}
\arrow[two heads]{l}
\arrow[hook]{r}
&
D'^{(1)}
\arrow{u}[swap]{s}
\arrow{d}{t}
\\
D^{(0)}
&
&
D'^{(0)}
\arrow[two heads]{ll}
\end{tikzcd} \]
of finite sets, where the restriction refers to the corresponding proper constructible bundle.  Specifically, the elements of
\begin{equation}
\label{edges.in.R.prime.projecting.to.vertices.in.R}
\left( D'^{(1)} \right)_{|D^{(0)}}
\cong
D'^{(1)}
\backslash
\left( D'^{(1)} \right)_{|D^{(1)}}
\end{equation}
correspond to edges in $D'$ that project to vertices in $D$.  The cocartesian monodromy in $\int_{|\D}\BV$ over this creation morphism is given by the composite
\[
\cV^{\times D^{(1)}}
\longra
\cV^{\times \left( D'^{(1)} \right)_{|D^{(1)}}}
\longra
\cV^{\times D'^{(1)}}
\]
of the pullback followed by the functor taking all elements of the set \Cref{edges.in.R.prime.projecting.to.vertices.in.R} to the unit object $\uno_\cV \in \cV$.  Then, a (necessarily cocartesian) morphism in the source $\int_{|\D} \cdiC$ over this morphism in $\D$ whose source is an object
\[
D^{(0)}
\xlongra{\lambda}
\cV
\]
must have target the object
\[
D'^{(0)}
\xra{\varphi^{(0)}}
D^{(0)}
\xlongra{\lambda}
\cV
~.
\]
Then, the composite \Cref{composite.giving.categorified.fact.hlgy.over.D} takes this to a morphism in $\int_{|\D} \BV$ which, after taking its cocartesian-active factorization, is given in the factor of $\cV$ corresponding to an element $e \in D'^{(1)}$ by
\begin{itemize}
\item the identity map of the object
\[
\ulhom_\cC \left( \lambda \left(\varphi^{(0)}(s(e)) \right),\lambda \left( \varphi^{(0)}(t(e)) \right) \right)
\]
if $e \in \left( D'^{(1)} \right)_{|D^{(1)}}$, and
\item the unit map
\[
\uno_\cV
\longra
\ulhom_\cC \left( \lambda \left(\varphi^{(0)}(s(e)) \right),\lambda \left( \varphi^{(0)}(t(e)) \right) \right)
\]
if $e \notin \left( D'^{(1)} \right)_{|D^{(1)}}$ (whose existence arises from the application of the map $\lambda$ to the identification
\[
\varphi^{(0)} ( s(e)) \cong \varphi^{(0)} ( t(e))
\]
in $D^{(0)}$).
\end{itemize}

\item A refinement morphism $D \ra D'$ in $\D$ determines a commutative diagram
\[ \begin{tikzcd}
D^{(0)}
\arrow[hookleftarrow]{r}
&
D'^{(0)}
\\
D^{(1)}
\arrow[two heads]{r}
\arrow{u}{s}
\arrow{d}[swap]{t}
&
D'^{(1)}
\arrow{u}[swap]{s}
\arrow{d}{t}
\\
D^{(0)}
\arrow[hookleftarrow]{r}
&
D'^{(0)}
\end{tikzcd} \]
of finite sets.  Now, a (necessarily cocartesian) morphism in the source $\int_{|\D} \cdiC$ over this morphism in $\D$ is taken by the composite \Cref{composite.giving.categorified.fact.hlgy.over.D} to a morphism in $\int_{|\D} \BV$ determined by the composition maps of $\cC$.
\end{itemize}
\end{remark}

\begin{observation}
\label{obs.cocart.over.cls.and.some.cr}
The composite morphism \Cref{composite.giving.categorified.fact.hlgy.over.D} (or more primitively, the morphism \Cref{fact.hlgy.over.D.of.unit}) in $\Cat_{\cocart/\D}$ actually preserves cocartesian morphisms over certain morphisms in $\D$: the closed morphisms, as well as the creation maps corresponding contravariantly to \textit{stratified covering spaces}.\footnote{Such a morphism $D \ra D'$ in $\D$ determines a diagram
\[ \begin{tikzcd}[ampersand replacement=\&]
D^{(0)}
\&
D'^{(0)}
\arrow{l}
\\
D^{(1)}
\arrow{u}{s}
\arrow{d}[swap]{t}
\&
D'^{(1)}
\arrow{l}
\arrow{u}[swap]{s}
\arrow{d}{t}
\\
D^{(0)}
\&
D'^{(0)}
\arrow{l}
\end{tikzcd} \]
of finite sets.  A (necessarily cocartesian) morphism in $\int_{|\D} \cdiC$ over this morphism in $\D$ must start at an object $D^{(0)} \ra \iC$ and end at the composite $D'^{(0)} \ra D^{(0)} \ra \iC$, and its image under the composite \Cref{composite.giving.categorified.fact.hlgy.over.D} will start at an object $D^{(1)} \ra \cV$ and likewise end at the composite $D'^{(1)} \ra D^{(1)} \ra \cV$.}  Note in particular that the latter class includes (the images in $\D$ of) the cocartesian morphisms in the cocartesian fibration $\epicyc \da \BW$.  Hence, when the composite morphism \Cref{composite.giving.categorified.fact.hlgy.over.D} in $\Cat_{\cocart/\D}$ is pulled back along the composite
\[
\epicyc
\longra
\totalD
\xlongra{s}
\D
~,
\]
it defines a morphism in $\coCart_\BW \subset \Cat_{\cocart/\BW}$.
\end{observation}

\section{Enriched factorization homology}
\label{section.enriched.fact.hlgy}

In this section, we study enriched factorization homology.  We begin in \Cref{subsection.s.m.cats} with some preliminaries on symmetric monoidal $\infty$-categories as well as a study of the functoriality of 1-dimensional strata on the $\infty$-category $\D$.  Then, in \Cref{subsection.enr.fact.hlgy} we construct enriched factorization homology over a fixed object $M \in \M$ (as described in \Cref{subsection.fact.hlgy.intro}) and equip it with a canonical action of its automorphism group.

\subsection{Preliminaries on symmetric monoidal $\infty$-categories}
\label{subsection.s.m.cats}

\begin{notation}
For a symmetric monoidal $\infty$-category $(\cV,\boxtimes)$, we write
\begin{equation}
\label{symm.mon.deloop}
\begin{tikzcd}
\BSVbox
\arrow{d}
\\
\Fin_*
\end{tikzcd}
\end{equation}
for its \textit{symmetric monoidal deloop}, i.e., its corresponding $\infty$-operad.  When the symmetric monoidal structure is understood, we simply write
\[
\BSV
:=
\BSVbox
~.
\]
\end{notation}

\begin{notation}
Recall that the functor \Cref{symm.mon.deloop} restricts to a cocartesian fibration when pulled back along the surjective monomorphism
\[
\Fin
\longrsurjmono
\Fin_*
\]
from the subcategory of \textit{active} morphisms.  Hence, the terminal maps in $\Fin$ collectively induce a \textit{total tensor product} functor to the fiber $\BSV_{|\pt} \simeq \cV$, which we denote by
\begin{equation}
\label{tensor.everything.together.over.Fin}
\BSV_{|\Fin}
\xlongra{\bigboxtimes}
\cV
~.
\end{equation}
\end{notation}

Whereas we can only generally take indexed tensor products in $\cV$ over morphisms in $\Fin$, the morphisms in $\D$ can have more exotic behavior on 1-dimensional strata.  This behavior is encoded as follows.

\begin{definition}
\label{define.span.fin}
We write
\[
\Corr(\Fin)
\]
for the $\infty$-category of \bit{correspondences of finite sets}: its objects are finite sets, a morphism from $S$ to $T$ is given by a span
\begin{equation}
\label{typical.morphism.in.Corr.Fin}
\begin{tikzcd}
&
U
\arrow{ld}[swap]{\varphi}
\arrow{rd}{\psi}
\\
S
&
&
T
\end{tikzcd}
\end{equation}
of finite sets, and composition is given by pullback.\footnote{This can be easily constructed e.g.\! as a complete Segal space, using the fact that $\Fin$ admits finite limits.}\footnote{In fact, $\Corr(\Fin)$ is a $(2,1)$-category: for any finite sets $S,T \in \Fin$ we have an equivalence $\hom_{\Corr(\Fin)}(S,T) \simeq (\Fin_{/(S \times T)})^\simeq$ of 1-groupoids.}
\end{definition}

\begin{observation}
There is a canonical composite
\[
\Fin
\longrsurjmono
\Fin_*
\longrsurjmono
\Corr(\Fin)
\]
of surjective monomorphisms, in which the first functor adds a disjoint basepoint and the second functor takes a map $S_+ \xra{\varphi} T_+$ to the span
\[ \begin{tikzcd}
&
\varphi^{-1}(T)
\arrow{ld}
\arrow{rd}
\\
S
&
&
T
\end{tikzcd} \]
given by the preimage of the subset $T \subset T_+$.
\end{observation}

\begin{remark}
Whereas a morphism in $\Fin_*$ might be thought of as a ``partially-defined function'' (on the objects of $\Fin$ obtained by removing basepoints), a morphism in $\Corr(\Fin)$ might be thought of as a ``partially-defined multi-valued function'' (which allows for the repetition of values).
\end{remark}

\begin{notation}
\label{notation.one.d.strata.defines.corr}
A morphism in $\D$, which is in particular a proper constructible bundle $M \da \Delta^1$, determines a span
\[
M_0^{(1)}
\simeq
\Gamma \left(
\begin{tikzcd}
M_0^{(1)}
\arrow{d}
\\
\Delta^{\{0\}}
\end{tikzcd} \right)
\longla
\Gamma' \left(
\begin{tikzcd}
M
\arrow{d}
\\
\Delta^1
\end{tikzcd} \right)
\longra
\Gamma \left(
\begin{tikzcd}
M_1^{(1)}
\arrow{d}
\\
\Delta^{\{1\}}
\end{tikzcd} \right)
\simeq
M_1^{(1)}
\]
of finite sets, where
\begin{itemize}
\item we write $\Gamma$ for the space of sections (among stratified spaces), and
\item we write $\Gamma'$ for the space of sections (also among stratified spaces) that take values in the 1-dimensional stratum of each fiber.
\end{itemize}
By \cite[Lemma 3.36]{AFR-fact}, this assembles into a functor
\begin{equation}
\label{Exit.one.on.D}
\D
\xra{(-)^{(1)}}
\Corr(\Fin)
~.
\end{equation}
\end{notation}

\begin{observation}
\label{action.of.one.d.stratum.functor}
The action of the functor $(-)^{(1)}$ on the various classes of morphisms in $\D$ is as follows.
\begin{itemize}
\item It takes closed morphisms to backwards monomorphisms.
\item It takes refinement morphisms to forwards surjections.
\item A creation morphism $D \ra D'$ corresponds contravariantly to a proper constructible bundle, which determines a ``partially-defined multi-valued function'' from $D^{(1)}$ to $D'^{(1)}$ by taking an edge of $D$ to its preimage in $D'$.\footnote{Note that edges in $D'$ can project to vertices in $D$, so that the forwards map in this span need not be surjective.}
\end{itemize}
\end{observation}

Absent any further assumptions on $(\cV,\boxtimes)$, the pullback
\[ \begin{tikzcd}
\D_{|\Fin}
\arrow{r}
\arrow[hook, two heads]{d}
&
\Fin
\arrow[hook, two heads]{d}
\\
\D
\arrow{r}
&
\Fin_*
\end{tikzcd} \]
dictates the functoriality of enriched factorization homology: it is the largest source on which we can expect a total tensor product functor.  This is encoded by the following result.

\begin{lemma}
\label{tensor.everything.together.from.int.BV.over.D.Fin}
For any symmetric monoidal $\infty$-category $\cV$, there is a canonical pullback square
\[ \begin{tikzcd}
\dint_{\! \! \! |\D_{|\Fin_*}} \BV
\arrow{r}
\arrow{d}
&
\BSV
\arrow{d}
\\
\D_{|\Fin_*}
\arrow{r}
&
\Fin_*
\end{tikzcd}~. \]
\end{lemma}

\begin{proof}
Consider the commutative diagram
\begin{equation}
\label{diagram.containing.bDelta.op.mapping.to.D.star}
\begin{tikzcd}[row sep=1.5cm, column sep=1.5cm]
\bDelta^\op
\arrow[dashed]{rd}[pos=0.55]{\varphi}
\arrow[bend left=0]{rrd}{\Delta^1/\partial\Delta^1}
\arrow[bend right=0]{rdd}[swap]{\rho}
\\
&
\D_{|\Fin_*}
\arrow{r}
\arrow[hook, two heads]{d}{\psi}
&
\Fin_*
\arrow[hook, two heads]{d}
\\
&
\D
\arrow{r}[swap]{(-)^{(1)}}
&
\Corr(\Fin)
\end{tikzcd}
~,
\end{equation}
in which the factorization arises from the universal property of the pullback.  Recall that by definition of the underlying monoidal $\infty$-category of a symmetric monoidal $\infty$-category, we have a pullback square
\[ \begin{tikzcd}[row sep=1.5cm, column sep=1.5cm]
\BV
\arrow{r}
\arrow{d}
&
\BSV
\arrow{d}
\\
\bDelta^\op
\arrow{r}[swap]{\Delta^1/\partial\Delta^1}
&
\Fin_*
\end{tikzcd}
~.
\]
From diagram \Cref{diagram.containing.bDelta.op.mapping.to.D.star}, we therefore obtain a cospan
\begin{equation}
\label{cospan.from.int.over.D.star.of.BV.to.BSV.over.D.star}
\int_{|\D_{|\Fin_*}} \BV
:=
\psi^* \rho_* \BV
\longra
\varphi_* \BV
\simeq
\varphi_* \varphi^* \left( \BSV \right)_{|\D_{|\Fin_*}}
\longla
\left( \BSV \right)_{|\D_{|\Fin_*}}
\end{equation}
in $\coCart_{\D_{|\Fin_*}}$, where the rightwards morphism arises from the universal property of the right Kan extension $\varphi_*$ while the leftwards morphism is the unit of the adjunction $\varphi^* \adj \varphi_*$.

We claim that cospan \Cref{cospan.from.int.over.D.star.of.BV.to.BSV.over.D.star} consists of equivalences.  Since it lies in $\coCart_{\D_{|\Fin_*}}$, it suffices to check this on fibers over an arbitrary object $D \in \D_{|\Fin_*}$, where we obtain a cospan
\begin{equation}
\label{cospan.from.int.R.BV.to.Rone.copies.of.V}
\int_D \BV
\longra
(\varphi_* \BV)_{|D}
\longla
\cV^{\times D^{(1)}}
~.
\end{equation}
The fact that the rightwards functor in cospan \Cref{cospan.from.int.R.BV.to.Rone.copies.of.V} is an equivalence follows by combining~\Cref{t18}.  
Moreover, the equivalence between the outer two terms in cospan \Cref{cospan.from.int.R.BV.to.Rone.copies.of.V} of \Cref{t18} is clearly compatible with the cospan.\footnote{This is just the assertion that the functor $D^{(1)} \ra \Enter(D)$ induces an equivalence on limits with respect to the long composite in the commutative diagram
\[ \begin{tikzcd}[ampersand replacement=\&]
\&
(\bDelta^\op)_{D/^*}
\arrow[hook]{dd}
\arrow{rd}
\\
\Enter(D)
\arrow{ru}[sloped, pos=0.8]{~}
\arrow{rd}[sloped, swap, pos=0.8]{~}
\&
\&
\bDelta^\op
\arrow{r}{\BV}
\&
\Cat_1
\\
\&
(\bDelta^\op)_{D/}
\arrow{ru}
\end{tikzcd}~, \]
which is clear in light of the simplicity of the category $\Enter(D)$ and the functor $\bDelta^{\op,\cls}_{\leq 1} \hookra \bDelta^\op \xra{\BV} \Cat_1$ through which this composite factors.}  So cospan \Cref{cospan.from.int.R.BV.to.Rone.copies.of.V} consists of equivalences, which proves the claim.
\end{proof}

\begin{notation}
For any functor $\cK \ra \D_{|\Fin}$, we write
\[ \begin{tikzcd}[column sep=1.5cm, row sep=1.5cm]
\dint_{\! \! \! |\cK} \BV
\arrow{r}
\arrow{d}
\arrow[dashed, bend left]{rr}{\boxtimesfortikz}
&
\dint_{\! \! \! |\D_{|\Fin}} \BV
\arrow{r}{\boxtimesfortikz}
\arrow{d}
&
\cV
\\
\cK
\arrow{r}
&
\D_{|\Fin}
\end{tikzcd} \]
for the total tensor product functor that results from \Cref{tensor.everything.together.from.int.BV.over.D.Fin}.
\end{notation}

\begin{remark}
As we will show in \Cref{subsection.cart.symm.mon}, when $(\cV,\boxtimes) = (\cV,\times)$ is \textit{cartesian} symmetric monoidal, then there exists a \textit{total cartesian product} functor indexed over all of $\Corr(\Fin)$.\footnote{In fact, definitionally, this hypothesis can be weakened (see \Cref{rmk.bicomm.bialg.in.Cat}).}  This induces a total cartesian product functor over all of $\D$, which as we will show in \Cref{fact.hlgy.with.cartesian.enr} extends the functoriality of enriched factorization homology in this case.
\end{remark}

\subsection{Enriched factorization homology}
\label{subsection.enr.fact.hlgy}

\begin{observation}
\label{obs.D.of.M.factors.through.Fin}
For any stratified stratified space $M \in \M$, from 
Observations 
\Cref{obs.possible.maps.of.refinements} 
\and \Cref{action.of.one.d.stratum.functor} we see that the composite
\[
\D(M)
\longra
\totalD
\xlongra{s}
\D
\xra{(-)^{(1)}}
\Corr(\Fin)
\]
factors through the subcategory $\Fin \subset \Corr(\Fin)$; it follows that there exists a factorization
\[ \begin{tikzcd}
\D(M)
\arrow[dashed]{r}
\arrow{d}
&
\D_{|\Fin}
\arrow[hook, two heads]{d}
\\
\totalD
\arrow{r}[swap]{s}
&
\D
\end{tikzcd}~. \]
\end{observation}

\begin{definition}
\label{define.enr.fact.hlgy}
Let $M \in \M$ be a stratified stratified space, let $\cV$ be a symmetric monoidal $\infty$-category, and let $\cC \in \fCat_1(\cV)$ be a flagged $\cV$-enriched $(\infty,1)$-category.  Then, we define the (\bit{enriched}) \bit{factorization homology} of $\cC$ over $M$ to be the colimit
\[
\int_M \cC
:=
\colim
\left(
\int_{|\D(M)} \cdiC
\xra{\int_{|\D(M)} \ulhom_\cC}
\int_{|\D(M)} \BV
\xlongra{\bigboxtimes}
\cV
\right)
~.
\]
This assembles into a functor
\begin{equation}
\label{fact.hlgy.over.M.as.a.functor}
\fCat_1(\cV)
\xra{\int_M(-)}
\cV
~.
\end{equation}
\end{definition}

\begin{remark}
\label{rmk.only.cyclically.monoidal}
Depending on the combinatorics of $\D(M)$, it can be possible to weaken the hypothesis that $\cV$ be symmetric monoidal while still obtaining a total tensor product functor
\[
\int_{|\D(M)} \BV
\xlongra{\bigboxtimes}
\cV
~.
\]
For instance, if $M = \SS^1$ then it suffices (definitionally) for $\cV$ to be merely \textit{paracyclically} monoidal: to have a well-defined tensor product of any finite set of objects indexed by a configuration of points in $\SS^1$.  (This notion sits between $\EE_1$ and $\EE_2$.)
\end{remark}

\begin{remark}
There is a sense in which factorization homology is given by a colimit over factorization homology.  Namely, by \Cref{t18}, for any simplicial object $\cY \in \Fun(\bDelta^\op,\Cat_1)$ and any $M \in \M$, there is a canonical functor
\[
\int_{|\D(M)} \cY
\longra
\int_M \cY
\]
which is a localization with respect to the cocartesian morphisms over $\D(M)$.  By inspection, we see that for any $\cC \in \fCat_1(\cV)$ there exists an extension
\[ \begin{tikzcd}[column sep=2cm, row sep=1.5cm]
\dint_{\! \! \! |\D(M)} \cdiC
\arrow{r}{\int_{|\D(M)} \ulhom_\cC}
\arrow{d}
&
\dint_{\! \! \! |\D(M)}
\BV
\arrow{r}{\boxtimesfortikz}
&
\cV
\\
\dint_{\! \! \! M} \cdiC
\arrow[dashed, bend right=10]{rru}
\end{tikzcd} \]
along this localization, whose colimit therefore also computes the enriched factorization homology
\[
\int_M \cC
\]
since localizations are final.\footnote{On the other hand, there does not exist an extension
\[ \begin{tikzcd}[ampersand replacement=\&, column sep=2cm, row sep=1.5cm]
\dint_{\! \! \! |\D(M)} \cdiC
\arrow{r}{\int_{|\D(M)} \ulhom_\cC}
\arrow{d}
\&
\dint_{\! \! \! |\D(M)}
\BV
\arrow{d}
\\
\dint_{\! \! \! M} \cdiC
\arrow[dotted]{r}[description]{\not\exists}
\&
\dint_{\! \! \! M} \BV
\end{tikzcd} \]
before postcomposition with the total tensor product functor $\int_{|\D(M)} \BV \xra{\boxtimes} \cV$.  In general, cartesian factorization homology is only functorial for \textit{strict} (as opposed to right-lax) functors of category-objects, and in this case it is also easy to see directly that such an extension does not exist (except in the trivial case that $M \in \M$ is a finite disjoint union of copies of the 0-disk $\DD^0$).}
\end{remark}

It is easy to construct an action of the automorphism ($\infty$-)group
\[
\Aut(M)
:=
\Aut_\M(M)
\]
of the object $M \in \M$ on the factorization homology $\int_M \cC$.

\begin{observation}
The pullback
\begin{equation}
\label{disk.stratns.mod.aut}
\begin{tikzcd}
\D(M)_{\htpy \Aut(M)}
\arrow[hook]{r}
\arrow{d}
&
\totalD
\arrow{d}{t}
\\
\sB \Aut(M)
\arrow[hook]{r}
&
\M
\end{tikzcd}
\end{equation}
is a cocartesian fibration: disk-refinements are functorial for automorphisms.
\end{observation}

\begin{remark}
In the case that $M = \SS^1$ is the smooth circle, the pullback \Cref{disk.stratns.mod.aut} recovers the functor
\[ \begin{tikzcd}
\cyclic
\arrow{d}
\\
\BT
\end{tikzcd} \]
appearing in~\Cref{t10}, the canonical functor from $\cyclic$ to its groupoid completion (namely the classifying space of the circle group $\TT$). 
\end{remark}

\begin{observation}
\label{obs.D.of.M.quotiented.by.Aut.M.factors.through.Fin}
Extending \Cref{obs.D.of.M.factors.through.Fin}, for any stratified stratified space $M \in \M$, there exists a factorization
\[ \begin{tikzcd}
\D(M)_{\htpy \Aut(M)}
\arrow[dashed]{r}
\arrow{d}
&
\D_{|\Fin}
\arrow[hook, two heads]{d}
\\
\totalD
\arrow{r}[swap]{s}
&
\D
\end{tikzcd}~. \]
\end{observation}

\begin{construction}
\label{construct.aut.action.on.enr.fact.hlgy}
The left Kan extension
\begin{equation}
\label{LKE.gives.AutM.action}
\begin{tikzcd}[column sep=3cm, row sep=1.5cm]
\dint_{\! \! \! |\D(M)_{\htpy \Aut(M)}} \cdiC
\arrow{r}{\int_{|\D(M)_{\htpy \Aut(M)}} \ulhom_\cC}
\arrow{d}
&
\dint_{\! \! \! |\D(M)_{\htpy \Aut(M)}} \BV
\arrow{ld}
\arrow{r}{\boxtimesfortikz}
&
\cV
\\
\D(M)_{\htpy \Aut(M)}
\arrow{d}
\\
\sB\Aut(M)
\arrow[dashed, bend right=10]{rruu}[swap]{\Aut(M) \lacts \int_M \cC}
\end{tikzcd}
\end{equation}
can be computed as a fiberwise colimit since it is along a cocartesian fibration.  As the fibers of the vertical composite are given by
\[ \begin{tikzcd}[row sep=1.5cm]
\dint_{\! \! \! |\D(M)} \cdiC
\arrow{r}
\arrow{d}
&
\dint_{\! \! \! |\D(M)_{\htpy \Aut(M)}} \cdiC
\arrow{d}
\\
\D(M)
\arrow{r}
\arrow{d}
&
\D(M)_{\htpy \Aut(M)}
\arrow{d}
\\
\{ M \}
\arrow{r}
&
\sB\Aut(M)
\end{tikzcd}~, \]
we see that the underlying object in the left Kan extension \Cref{LKE.gives.AutM.action} is indeed $\int_M \cC$, as indicated.  This therefore constructs a canonical action of $\Aut(M)$ on $\int_M \cC$, and thereafter a lift
\[ \begin{tikzcd}[row sep=1.5cm]
\fCat_1(\cV)
\arrow[dashed]{r}
\arrow{rd}[swap]{\int_M(-)}
&
\Fun(\sB\Aut(M),\cV)
\arrow{d}{\fgt}
\\
&
\cV
\end{tikzcd} \]
of the factorization homology over $M$ functor \Cref{fact.hlgy.over.M.as.a.functor}.
\end{construction}

\begin{definition}
For any flagged $\cV$-enriched $(\infty,1)$-category $\cC$, we define its (\bit{$\cV$-enriched}) \bit{topological Hochschild homology} to be its enriched factorization homology over the smooth circle $\SS^1 \in \M$; we denote this by
\[
\THH_\cV(\cC)
:=
\left( \int_{\SS^1} \cC \right)
\in \cV
~,
\]
and the construction assembles into a functor
\[
\fCat_1(\cV)
\xra{\THH_\cV}
\cV
~.
\]
\end{definition}

\begin{remark}
As a particular case of \Cref{construct.aut.action.on.enr.fact.hlgy}, for any $\cV$-enriched $(\infty,1)$-category $\cC$ we obtain a natural action
\[
\TT
\lacts
\THH_\cV(\cC)
\]
of the circle group $\TT$ on its $\cV$-enriched topological Hochschild homology.  In the special case that $(\cV,\boxtimes) = (\cV,\times)$ is \textit{cartesian} symmetric monoidal, in~\cite{circle} we enhance this $\TT$-action to an \textit{unstable cyclotomic structure} on $\THH_\cV(\cC)$.
\end{remark}

\begin{observation}
\label{obs.fact.hlgy.over.D.zero}
There is a canonical identification
\[
\int_{\DD^0} \cC
\simeq
\iC
\tensoring
\uno_\cV
\]
of the enriched factorization homology of $\cC$ over $\DD^0$ with the tensoring of the unit object $\uno_\cV \in \cV$ over the underlying $\infty$-groupoid of $\cC$.
\end{observation}

As in \cite[Proposition 4.1.8]{circle}, enriched factorization homology satisfies \bit{excision}.
\begin{prop}[Excision]
\label{excision}
Let $\cC \in \fCat_1(\cV)$ be a flagged $\cV$-enriched $(\infty,1)$-category $\cC$.
Let $M \in \M$, and let $S \subset M^{(1)}$ be a finite subset of the 1-dimensional strata of $M$ such that each circle-component of $M$ contains an element in $S$.
Via \cite[Construction 3.7.1]{circle}, $S$ determines a functor
\[
\bDelta^{\op}
\xra{~D_\bullet~}
\D(M)
\]
whose colimit in $\M$ exists and is $M$.  
Furthermore, the canonical morphism in $\cV$
\[
\underset{[p]^\circ  \in \bDelta^{\op}}\colim
\Bigl(
\int_{D_p} \cC
\Bigr)
~=~
\Bigl|
\int_{D_\bullet} \cC
\Bigr|
\xra{~\simeq~}
\int_M \cC
\]
is an equivalence.

\end{prop}

\begin{example}
\Cref{excision} is a \bit{local-to-global formula} for enriched factorization homology.  
It gives a method to compute the value $\int_M \cC$ of enriched factorization homology on a general object $M\in \M$ in terms of simpler values.
Here are some such examples.
\begin{enumerate}

\item
In the case that $S \subset \SS^1$ is a subset of cardinality 2, 
\Cref{excision} implies the canonical morphism induced by the composition rule of $\cC$,
\[
\circ
\colon
\hom_\cC(\bullet,\bullet')
\underset{ (\bullet^\circ , \bullet') \in \cC^{\op} \times \cC} \bigotimes
\hom_\cC(\bullet',\bullet)
\xra{~\simeq~}
\THH_\cV(\cC)
~,
\]
is an equivalence,
which witnesses the $\cV$-enriched topological Hochschild homology of $\cC$ as the coend of the identity $(\cC,\cC)$-bimodule $\cC$ with itself.

\item
In the case that $S \subset \SS^1$ is a singleton subset, 
\Cref{excision} implies the canonical morphism induced by the composition rule of $\cC$:
\[
\circ
\colon
\Bigl|
{\sf Rep}_\cC( \chi_\bullet)
\Bigr|
\xra{~\simeq~}
\THH_\cV(\cC)
\]
is an equivalence.  
Here, the cosimplicial quiver
$
\chi_\bullet \colon 
\bDelta 
\to 
\Quiv
$ 
evaluates on $[p]$ as the cyclically-directed quiver whose (cyclically-directed) set of vertices is $\{0,1,\dots,p\}$.
\end{enumerate}

\end{example}

\section{Cartesian enriched factorization homology}
\label{section.cartesian.enr.fact.hlgy}

In this section, we study enriched factorization homology as defined in \Cref{section.enriched.fact.hlgy} in the special case that our enriching symmetric monoidal $\infty$-category $(\cV,\boxtimes) = (\cV,\times)$ is cartesian.  In this case, enriched factorization homology becomes functorial over all of $\M$, as we show in \Cref{fact.hlgy.with.cartesian.enr}.  In order to prove this, we extend the cartesian symmetric monoidal deloop $\BSVx \da \Fin_*$ to a larger object $\wBSVx \da \Corr(\Fin)$ in \Cref{subsection.cart.symm.mon}.

\subsection{Preliminaries on cartesian symmetric monoidal $\infty$-categories}
\label{subsection.cart.symm.mon}

The following result codifies the additional structure present on \textit{cartesian} symmetric monoidal $\infty$-categories.

\begin{prop}
\label{all.about.extended.cart.s.m.deloop}
Let $\cV$ be an $\infty$-category admitting finite products.

\begin{enumerate}

\item\label{cart.s.m.deloop.exists}
There exists a canonical cocartesian fibration
\[ \begin{tikzcd}
\wBSVx
\arrow{d}
\\
\Corr(\Fin)
\end{tikzcd} \]
satisfying the following properties.
\begin{enumeratesub}
\item\label{pb.of.extended.cart.s.m.deloop.is.cart.s.m.deloop}
Its pullback along the surjective monomorphism $\Fin_* \rsurjmono \Corr(\Fin)$ is the symmetric monoidal deloop
\[ \begin{tikzcd}
\BSVx
\arrow{r}
\arrow{d}
&
\wBSVx
\arrow{d}
\\
\Fin_*
\arrow[hook, two heads]{r}
&
\Corr(\Fin)
\end{tikzcd} \]
of the cartesian symmetric monoidal $\infty$-category $(\cV,\times)$; in particular, its fiber over an object $S \in \Corr(\Fin)$ is canonically identified with the indexed product $\cV^{\times S} \simeq \Fun(S,\cV)$.
\item\label{cocart.mdrmy.of.extended.cart.s.m.deloop}
Its cocartesian monodromy along a morphism \Cref{typical.morphism.in.Corr.Fin} in $\Corr(\Fin)$ is the composite
\begin{equation}
\label{eqn.cocart.mdrmy.of.extended.cart.s.m.deloop}
\begin{tikzcd}[row sep=0cm]
\cV^{\times S}
\arrow{r}
&
\cV^{\times U}
\arrow{r}
&
\cV^{\times T}
\\
\rotatebox{90}{$\in$}
&
\rotatebox{90}{$\in$}
&
\rotatebox{90}{$\in$}
\\
\left(V_s\right)_{s \in S}
\arrow[maps to]{r}
&
\left(V_{\varphi(u)}\right)_{u \in U}
\arrow[maps to]{r}
&
\left(
\prod_{u \in \psi^{-1}(t)}
V_{\varphi(u)}
\right)_{t \in T}
\end{tikzcd}
\end{equation}
of pullback along the map $S \xla{\varphi} U$ followed by the indexed product along the map $U \xra{\psi} T$.
\end{enumeratesub}

\item\label{product.everything.together.functor.on.extended.cart.s.m.deloop}
The total cartesian product functor \Cref{tensor.everything.together.over.Fin} for $(\cV,\times)$ admits a canonical extension
\[ \begin{tikzcd}[column sep=1.5cm]
\BSV_{|\Fin}
\arrow[hook, two heads]{r}
\arrow[bend left]{rrr}{\prod}
&
\BSV
\arrow[hook, two heads]{r}
&
\wBSVx
\arrow[dashed]{r}{\prod}
&
\cV
\end{tikzcd}~. \]

\item\label{extended.cart.s.m.deloop.computes.catfied.fact.hlgy}
There is a canonical pullback square
\[ \begin{tikzcd}
\dint_{\! \! \! |\D} \BV
\arrow{r}
\arrow{d}
&
\wBSV
\arrow{d}
\\
\D
\arrow{r}[swap]{(-)^{(1)}}
&
\Corr(\Fin)
\end{tikzcd}~. \]

\end{enumerate}
\end{prop}

\begin{remark}
\label{rmk.bicomm.bialg.in.Cat}
For a general symmetric monoidal $\infty$-category $(\cV,\boxtimes)$, one might define an extension to a \textit{bicommutative bialgebra} structure on $\cV$ to be the data of a cocartesian fibration
\[
\wBSVbox \da \Corr(\Fin)
\]
which extends its symmetric monoidal deloop
\[ \BSVbox \da \Fin_* \]
as in \Cref{all.about.extended.cart.s.m.deloop} (but perhaps allowing for different and more exotic cocartesian monodromy than indexed monoidal products as in the composite \Cref{eqn.cocart.mdrmy.of.extended.cart.s.m.deloop}).  An extended total tensor product functor might then be furnished by an extension
\[ \begin{tikzcd}
\Corr(\Fin)
\arrow{r}{\wBSVbox}
\arrow[hook, two heads]{d}
&
\Cat_1
\\
\ul{\Corr(\Fin)}^{2\textup{-}\op}
\arrow[dashed]{ru}
\end{tikzcd} \]
of its unstraightening -- which might be called an \textit{augmentation} --, where $\ul{\Corr(\Fin)}$ denotes the evident enhancement of $\Corr(\Fin)$ to a $2$-category with
\[
\ulhom_{\ul{\Corr(\Fin)}}(S,T)
\simeq
\Fin_{/( S \times T) }
~:
\]
given a morphism \Cref{typical.morphism.in.Corr.Fin} in $\Corr(\Fin)$, this provides a natural transformation in the diagram
\[ \begin{tikzcd}
\cV^{\times S}
\arrow{r}[swap, transform canvas={xshift=0.9cm, yshift=-0.3cm}]{\Rightarrow}
\arrow{rd}
&
\cV^{\times U}
\arrow{r}
&
\cV^{\times T}
\arrow{ld}
\\
&
\cV
\end{tikzcd} \]
whose components are ``diagonal'' maps in $\cV$.  We expect this to suffice to extend the functoriality of $\cV$-enriched factorization homology, as we achieve in \Cref{fact.hlgy.with.cartesian.enr} via \Cref{all.about.extended.cart.s.m.deloop} in the cartesian symmetric monoidal case.
\end{remark}

The remainder of this subsection is devoted to the proof of \Cref{all.about.extended.cart.s.m.deloop}.

\begin{definition}
\label{define.cart.dual}
We define the \bit{cartesian dual} of an exponentiable fibration
\[ \left( \cF \da \cB^\op \right) \in \EFib_{\cB^\op} \]
to be the object 
\[ \left( \cF^\cartdual \da \cB \right) \in \Cat_{/\cB} \]
satisfying the universal property that for any test object $(\cK \da \cB) \in \Cat_{/\cB}$, the space of lifts
\[ \begin{tikzcd}
&
\cF^\cartdual
\arrow{d}
\\
\cK
\arrow{r}
\arrow[dashed]{ru}
&
\cB
\end{tikzcd} \]
is equivalent to the space of lifts
\[ \begin{tikzcd}
\TwAr(\cK)^\op
\arrow[dashed]{r}
\arrow{d}[swap]{t^\op}
&
\cF
\arrow{d}
\\
\cK^\op
\arrow{r}
&
\cB^\op
\end{tikzcd} \]
which take all cartesian morphisms for the right fibration $t^\op$ to cartesian morphisms in $\cF$ over $\cB^\op$.
\end{definition}

\begin{remark}
In the case that
\[ (\cF \da \cB^\op) \in \Cart_{\cB^\op} \subset \EFib_{\cB^\op} \]
is a \textit{cartesian} fibration, by \cite{BGN-dual} its cartesian dual in the sense of \Cref{define.cart.dual} is indeed its cartesian dual in the sense of item \Cref{fibrationconventions} of \Cref{subsection.notation.and.conventions}.
\end{remark}

\begin{observation}
\label{obs.rke.condition}
Let $(\cE^\op \da \cB^\op) \in \EFib_{\cB^\op}$ be an exponentiable fibration, and let $\cZ \in \Cat_1$ be an arbitrary $\infty$-category.  For any test object $(\cK \da \cB) \in \Cat_{/\cB}$, we will describe an equivalent datum to a lift
\begin{equation}
\label{lift.to.cart.dual.of.rel.fun}
\begin{tikzcd}
&
\Fun^\rel_{/\cB^\op}(\cE^\op,\ul{\cZ})^\cartdual
\arrow{d}
\\
\cK
\arrow{r}
\arrow[dashed]{ru}
&
\cB
\end{tikzcd}
\end{equation}
to the cartesian dual of the relative functor $\infty$-category.  For any object $k \in \cK$, consider the diagram
\[ \begin{tikzcd}[row sep=1.5cm, column sep=1.5cm]
\left( \{ k \}
\underset{\cB}{\times}
\cE
\right)^\op
\arrow{r}
\arrow{d}
&
\left( \cK_{k/}
\underset{\cB}{\times}
\cE
\right)^\op
\arrow{r}
\arrow{d}
&
\left(
\TwAr(\cK)
\underset{\cB}{\times}
\cE
\right)^\op
\arrow{r}
\arrow{d}
&
\left(
\cK
\underset{\cB}{\times}
\cE
\right)^\op
\arrow{r}
\arrow{d}
&
\cE^\op
\arrow{d}
\\
\{ (k \ra k)^\circ \}
\arrow{r}
&
\left( \cK_{k/} \right)^\op
\arrow{r}
\arrow{d}
&
\TwAr(\cK)^\op
\arrow{r}[swap]{t^\op}
\arrow{d}{s^\op}
&
\cK^\op
\arrow{r}
&
\cB^\op
\\
&
\{ k \}
\arrow{r}
&
\cK
\end{tikzcd} \]
in which all squares are pullbacks.  Then, unwinding the definitions, we see that a lift \Cref{lift.to.cart.dual.of.rel.fun} is equivalently given by a functor
\[
\left(
\TwAr(\cK)
\underset{\cB}{\times}
\cE
\right)^\op
\longra
\cZ
\]
such that for each $k \in \cK$, the resulting commutative triangle
\[ \begin{tikzcd}[column sep=1.5cm]
\left( \{ k \}
\underset{\cB}{\times}
\cE
\right)^\op
\arrow{r}
\arrow{d}
&
\cZ
\\
\left( \cK_{k/}
\underset{\cB}{\times}
\cE
\right)^\op
\arrow{ru}
\end{tikzcd} \]
is a right Kan extension.  We will refer to this as the \textit{right Kan extension condition} in the proof of \Cref{all.about.extended.cart.s.m.deloop}\Cref{product.everything.together.functor.on.extended.cart.s.m.deloop}.
\end{observation}

\begin{example}
\label{ex.walking.arrow.point.of.cart.dual.of.rel.fun}
Let us study the special case of \Cref{obs.rke.condition} when $\cK = [1]$.  It suffices to study the universal case where the map $\cK \xra{\sim} \cB$ is an equivalence; however, for clarity regarding the general case (where $\cB$ is arbitrary), we write $\cE_{|[1]} := \cE$ where appropriate.

First of all, note the identification
\[
\TwAr([1])^\op
\simeq
\left(
\begin{tikzcd}
&
01
\arrow{ld}
\arrow{rd}
\\
00
&
&
11
\end{tikzcd}
\right)
~,
\]
under which the right fibration $\TwAr([1])^\op \xra{t^\op} [1]^\op$ is given by collapsing the rightwards morphism; the leftwards morphism is its unique nonidentity cartesian morphism.  Then, an arbitrary lift
\begin{equation}
\label{lift.TwAr.of.walking.arrow.to.rel.fun}
\begin{tikzcd}
\TwAr([1])^\op
\arrow[dashed]{r}
\arrow{d}[swap]{t^\op}
&
\Fun^\rel_{/[1]^\op}(\cE^\op,\ul{\cZ})
\arrow{d}
\\
{[1]^\op}
\arrow{r}[swap]{\sim}
&
{[1]^\op}
\end{tikzcd}
\end{equation}
is equivalent data to an arbitrary functor
\[
\left( \cE^\op \right)_{|[1]^\op}
\coprod_{( \cE^\op)_{|1^\circ} }
\left(
( \cE^\op)_{|1^\circ}
\times
[1]
\right)
\longra
\cZ
~.
\]
Moreover, the lift \Cref{lift.TwAr.of.walking.arrow.to.rel.fun} takes the unique nonidentity cartesian morphism in $\TwAr([1])^\op$ to a cartesian morphism in $\Fun^\rel_{/[1]^\op}(\cE^\op,\ul{\cZ})$ precisely if in the commutative diagram
\[ \begin{tikzcd}
(\cE^\op)_{|0^\circ}
\arrow[hook]{d}
\arrow[bend left=20]{rrd}
\\
(\cE^\op)_{|[1]^\op}
\arrow[hook]{r}
\arrow[hookleftarrow]{d}
&
(\cE^\op)_{|[1]^\op}
\coprod_{(\cE^\op)_{|1^\circ}}
\left(
(\cE^\op)_{|1^\circ}
\times
[1]
\right)
\arrow{r}
&
\cZ
\\
(\cE^\op)_{|1^\circ}
\arrow[bend right=20]{rru}
\end{tikzcd} \]
the upper triangle is a right Kan extension.  From here, it follows that a lift \Cref{lift.TwAr.of.walking.arrow.to.rel.fun} which preserves cartesian morphisms, i.e., the datum of a lift
\[ \begin{tikzcd}
&
\Fun^\rel_{/[1]^\op}(\cE^\op,\ul{\cZ})^\cartdual
\arrow{d}
\\
{[1]}
\arrow{r}[swap]{\sim}
\arrow[dashed]{ru}
&
{[1]}
\end{tikzcd}~, \]
can be equivalently specified by a diagram
\[ \begin{tikzcd}[row sep=1.25cm, column sep=2cm]
(\cE^\op)_{|0^\circ}
\arrow{rd}[sloped, pos=0.45]{F}
\arrow[hook]{d}[swap]{i}
\\
(\cE^\op)_{|[1]^\op}
\arrow[hookleftarrow]{d}[swap]{j}
\arrow{r}{i_*F}
&
\cZ
\\
(\cE^\op)_{|1^\circ}
\arrow{ru}[sloped, pos=0.65]{j^*i_*F}[swap, sloped, transform canvas={xshift=-0.05cm, yshift=-0.15cm}]{\rotatebox{-90}{$\Rightarrow$}}
\arrow[bend right]{ru}[swap, sloped, pos=0.4]{G}
\end{tikzcd} \]
in which the upper and middle triangles commute and the upper triangle is a right Kan extension.\footnote{As the functor $i$ is fully faithful, if the upper triangle is a right Kan extension then it necessarily commutes.}
\end{example}

\begin{example}
Building on \Cref{ex.walking.arrow.point.of.cart.dual.of.rel.fun}, we immediately find that for an arbitrary exponentiable fibration $(\cE^\op \da \cB^\op) \in \EFib_{\cB^\op}$ and any $\cZ \in \Cat_1$, the datum of a lift
\[ \begin{tikzcd}
&
\Fun^\rel_{/\cB^\op}(\cE^\op,\ul{\cZ})^\cartdual
\arrow{d}
\\
{[n]}
\arrow{r}
\arrow[dashed]{ru}
&
\cB
\end{tikzcd} \]
is equivalent to the data of a list
\[
\left(
F_0
~,~
(j_{0,1})^*(i_{0,1})_* F_0
\longra
F_1
~,~
(j_{1,2})^*(i_{1,2})_* F_1
\longra
F_2
~,~
\ldots
~,~
(j_{n-1,n})^*(i_{n-1,n})_* F_{n-1}
\longra
F_n
\right)
~,
\]
where for each $k \in [n]^\op$ we have chosen an arbitrary functor
\[
(\cE^\op)_{|k^\circ}
\xra{F_k}
\cZ
\]
and for $1 \leq k \leq n$ we write
\[ \begin{tikzcd}[column sep=1.2cm]
(\cE^\op)_{|(k-1)^\circ}
\arrow[hook]{r}{i_{k-1,k}}
&
(\cE^\op)_{|\{k-1 < k \}^\op}
\arrow[hookleftarrow]{r}{j_{k-1,k}}
&
(\cE^\op)_{|k^\circ}
\end{tikzcd} \]
for the inclusions of the fibers.
\end{example}

\begin{notation}
We write
\[ \begin{tikzcd}
\ol{\Corr}(\Fin)
\arrow{d}
\\
\Corr(\Fin)
\end{tikzcd} \]
for the universal exponentiable fibration whose sections are all finite sets.\footnote{That is, $\Corr(\Fin)$ classifies such exponentiable fibrations, with the equivalence given by pulling back $\ol{\Corr}(\Fin)$.}  Explicitly, the fiber over an object $S \in \Corr(\Fin)$ is the set $S$, and a morphism in $\ol{\Corr}(\Fin)$ from $(S,s)$ to $(T,t)$ is given by a span
\[ \begin{tikzcd}
&
(U,u)
\arrow{ld}[swap]{\varphi}
\arrow{rd}{\psi}
\\
(S,s)
&
&
(T,t)
\end{tikzcd} \]
in $\Fin_*$.
\end{notation}

\begin{definition}
Let $\cV$ be an $\infty$-category admitting finite products.  Its \bit{extended cartesian symmetric monoidal deloop} is the $\infty$-category
\[ \begin{tikzcd}
\wBSVx
:=
\hspace{-1cm}
&
\Fun^\rel_{/\Corr(\Fin)^\op}
\left( \ol{\Corr}(\Fin)^\op , \ul{\cV} \right)^\cartdual
\arrow{d}
\\
&
\Corr(\Fin)
\end{tikzcd} \]
over $\Corr(\Fin)$.
\end{definition}

We now prove \Cref{all.about.extended.cart.s.m.deloop}; for clarity, we separate the proofs of its three parts.

\begin{proof}[Proof of \Cref{all.about.extended.cart.s.m.deloop}\Cref{cart.s.m.deloop.exists}]
Consider the morphism \Cref{typical.morphism.in.Corr.Fin} in $\Corr(\Fin)$; this induces a pullback diagram
\[ \begin{tikzcd}
S \underset{U}{\join} T
\arrow{r}
\arrow{d}
&
\ol{\Corr}(\Fin)
\arrow{d}
\\
{[1]}
\arrow{r}
&
\Corr(\Fin)
\end{tikzcd}~. \]
In view of \Cref{ex.walking.arrow.point.of.cart.dual.of.rel.fun}, we see that the functor
\begin{equation}
\label{projection.from.wBSVx.to.Corr.Fin.in.proof.of.part.one}
\begin{tikzcd}
\wBSVx
\arrow{d}
\\
\Corr(\Fin)
\end{tikzcd}
\end{equation}
is a locally cocartesian fibration: its fiber over an object $S \in \Corr(\Fin)$ is given by the functor $\infty$-category
\[
\Fun(S^\op,\cV)
\simeq
\Fun(S,\cV)
~,
\]
and its cocartesian monodromy over the morphism \Cref{typical.morphism.in.Corr.Fin} is given by right Kan extension followed by restriction in the diagram
\[ \begin{tikzcd}[column sep=1.5cm]
S^\op
\arrow{rd}
\arrow[hook]{d}
\\
\left( S \underset{U}{\join} T \right)^\op
\arrow{r}
\arrow[hookleftarrow]{d}
&
\cV
\\
T^\op
\arrow{ru}
\end{tikzcd}~. \]
For each element $t \in T^\op$ we can identify
\[
S^\op
\underset{\left( S \underset{U}{\join} T \right)^\op}{\times}
\left( \left( S \underset{U}{\join} T \right)^\op \right)_{t/}
\simeq
\left(
S
\underset{\left( S \underset{U}{\join} T \right)}{\times}
\left( S \underset{U}{\join} T \right)_{/t}
\right)^\op
\simeq
\left( \psi^{-1}(t) \right)^\op
~,
\]
so that this cocartesian pushforward takes a functor
\[
S^\op
\xra{(V_s)_{s \in S^\op}}
\cV
\]
to the functor
\[ \begin{tikzcd}[row sep=0cm]
T^\op
\arrow{r}
&
\cV
\\
\rotatebox{90}{$\in$}
&
\rotatebox{90}{$\in$}
\\
t
\arrow[maps to]{r}
&
\left(
\prod_{u \in \psi^{-1}(t)}
V_{\varphi(u)}
\right)
\end{tikzcd}~. \]
These cocartesian monodromy functors evidently compose, so that the functor \Cref{projection.from.wBSVx.to.Corr.Fin.in.proof.of.part.one} is indeed a cocartesian fibration.  Finally, we identify the pullback of the extended cartesian symmetric monoidal deloop $\wBSVx$ along $\Fin_* \rsurjmono \Corr(\Fin)$ as the cartesian symmetric monoidal deloop $\BSVx$ by \cite[Corollary 2.4.1.8]{LurieHA}.
\end{proof}

\begin{proof}[Proof of \Cref{all.about.extended.cart.s.m.deloop}\Cref{product.everything.together.functor.on.extended.cart.s.m.deloop}]
We construct the functor
\[
\wBSVx
\xlongra{\prod}
\cV
\]
as a natural transformation of complete Segal spaces.  By definition, a functor $[n] \ra \wBSVx$ is the data of an exponentiable fibration $\cE \da [n]$ along with a functor
\begin{equation}
\label{functor.from.iterated.pullback.in.proof.of.product.everything.together.functor.for.extended.cart.s.m.deloop}
\begin{tikzcd}
\left( \TwAr([n]) \underset{[n]}{\times} \cE \right)^\op
\arrow{r}
\arrow{d}
\arrow[dashed, bend left]{rrr}
&
\cE^\op
\arrow{r}
\arrow{d}
&
\ol{\Corr}(\Fin)^\op
\arrow{d}
&
\cV
\\
\TwAr([n])^\op
\arrow{r}[swap]{t^\op}
&
{[n]^\op}
\arrow{r}
&
\Corr(\Fin)^\op
\end{tikzcd}
\end{equation}
from the interated pullback that satisfies the right Kan extension condition of \Cref{obs.rke.condition}.  To this $[n]$-point of $\wBSVx$, we assign the $[n]$-point of $\cV$ given by the right Kan extension
\begin{equation}
\label{rkan.for.n.point.of.V}
\begin{tikzcd}[column sep=2cm]
\left( \TwAr([n]) \underset{[n]}{\times} \cE \right)^\op
\arrow{r}
\arrow{d}
&
\cV
\\
\TwAr([n])^\op
\arrow{d}[swap]{s^\op}
\\
{[n]}
\arrow[dashed]{ruu}
\end{tikzcd}~.
\end{equation}

Observe that the vertical composite in diagram \Cref{rkan.for.n.point.of.V} is the opposite of the vertical composite in the diagram
\[ \begin{tikzcd}
\TwAr([n]) \underset{[n]}{\times} \cE
\arrow{r}
\arrow{d}
&
\TwAr([n])
\arrow{d}{(s,t)}
\\
{[n]^\op} \times \cE
\arrow{r}
\arrow{d}
&
{[n]^\op \times [n]}
\\
{[n]^\op}
\end{tikzcd}~, \]
in which the square is a pullback and the lower vertical functor is the projection.  This is the composite of two cocartesian fibrations, so it is again a cocartesian fibration.  Thus, the vertical composite in diagram \Cref{rkan.for.n.point.of.V} is a cartesian fibration.  It follows that that right Kan extension is given by fiberwise limit: it takes the object $i \in [n]$ to the limit
\[
\lim
\left(
\left(
[n]_{i/}
\underset{[n]}{\times}
\cE
\right)^\op
\longra
\left(
\TwAr([n])
\underset{[n]}{\times}
\cE
\right)^\op
\longra
\cV
\right)
~.
\]
But further, as by assumption our functor \Cref{functor.from.iterated.pullback.in.proof.of.product.everything.together.functor.for.extended.cart.s.m.deloop} satisfies the right Kan extension condition, since right Kan extensions compose it follows that the canonical morphism
\begin{equation}
\label{reduce.lim.over.slide.to.lim.over.fiber.in.prod.functor}
\lim
\left(
\left(
[n]_{|i}
\underset{[n]}{\times}
\cE
\right)^\op
\longra
\left(
\TwAr([n])
\underset{[n]}{\times}
\cE
\right)^\op
\longra
\cV
\right)
\xlongra{\sim}
\lim
\left(
\left(
[n]_{i/}
\underset{[n]}{\times}
\cE
\right)^\op
\longra
\left(
\TwAr([n])
\underset{[n]}{\times}
\cE
\right)^\op
\longra
\cV
\right)
\end{equation}
is an equivalence.  Through this reduction, setting $n=0$ we see that this right Kan extension does indeed act on the fiber over an object $S \in \Corr(\Fin)$ as the functor
\[
\left( \wBSVx \right)_{|S}
\simeq
\Fun(S^\op,\cV)
\xra{\prod}
\cV
~,
\]
namely the limit over $S^\op$.

We now show that this assignment respects the simplicial structure maps of the complete Segal spaces of $\wBSVx$ and $\cV$.  For this, suppose we are given a composite
\[
[m]
\xlongra{\rho}
[n]
\longra
\wBSVx
~.
\]
From the universal property of right Kan extension, taking right Kan extensions as in diagram \Cref{rkan.for.n.point.of.V} yields a diagram
\[ \begin{tikzcd}[row sep=0.25cm]
{[m]}
\arrow{rd}[swap, sloped, transform canvas={xshift=0cm, yshift=-0.2cm}]{\rotatebox{90}{$\Rightarrow$}}
\arrow{dd}[swap]{\rho}
\\
&
\cV
\\
{[n]}
\arrow{ru}
\end{tikzcd}~, \]
which it remains to show strictly commutes.  This amounts to showing, for each $i \in [m]$, that its component
\[
\lim
\left(
\left(
[m]_{i/}
\underset{[m]}{\times}
\cE
\right)^\op
\longra
\left(
\TwAr([m])
\underset{[m]}{\times}
\cE
\right)^\op
\longra
\cV
\right)
\longra
\lim
\left(
\left(
[n]_{\rho(i)/}
\underset{[n]}{\times}
\cE
\right)^\op
\longra
\left(
\TwAr([n])
\underset{[n]}{\times}
\cE
\right)^\op
\longra
\cV
\right)
\]
is an equivalence.  But through the equivalence \Cref{reduce.lim.over.slide.to.lim.over.fiber.in.prod.functor} this is identified with the morphism
\[
\lim
\left(
\left(
[m]_{|i}
\underset{[m]}{\times}
\cE
\right)^\op
\longra
\left(
\TwAr([m])
\underset{[m]}{\times}
\cE
\right)^\op
\longra
\cV
\right)
\longra
\lim
\left(
\left(
[n]_{|\rho(i)}
\underset{[n]}{\times}
\cE
\right)^\op
\longra
\left(
\TwAr([n])
\underset{[n]}{\times}
\cE
\right)^\op
\longra
\cV
\right)
~,
\]
which is tautologically an equivalence: these are the limits of identical diagrams.  Thus, we have indeed constructed a functor
\[
\wBSVx
\xlongra{\prod}
\cV
~.
\]
That this restricts to $\BSVx$ appropriately follows from \cite[Proposition 2.4.1.6]{LurieHA}.
\end{proof}

\begin{proof}[Proof of \Cref{all.about.extended.cart.s.m.deloop}\Cref{extended.cart.s.m.deloop.computes.catfied.fact.hlgy}]
Observe that we have a commutative diagram
\[ \begin{tikzcd}
&
\BSV
\arrow{rr}
\arrow{dd}
&
&
\wBSV
\arrow{dd}
\\
\BV
\arrow{ru}
\arrow{dd}
\\
&
\Fin_*
\arrow{rr}
&
&
\Corr(\Fin)
\\
\bDelta^\op
\arrow{ru}[sloped, pos=1]{\Delta^1/\partial \Delta^1}
\arrow[hook]{rr}
&
&
\D
\arrow{ru}[swap, sloped, pos=0.2]{(-)^{(1)}}
\end{tikzcd} \]
in which both upper squares are pullbacks -- the left square by definition of the underlying nonsymmetric $\infty$-operad of an $\infty$-operad, and the right square by part \Cref{cart.s.m.deloop.exists}\Cref{pb.of.extended.cart.s.m.deloop.is.cart.s.m.deloop}.  The unit of the adjunction
\[ \begin{tikzcd}[column sep=1.5cm]
\coCart_\D
\simeq
\Fun(\D,\Cat_1)
\arrow[transform canvas={yshift=0.9ex}]{r}{\rho^*}
\arrow[hookleftarrow, transform canvas={yshift=-0.9ex}]{r}[yshift=-0.2ex]{\bot}[swap]{\rho_*}
&
\Fun(\bDelta^\op,\Cat_1)
\simeq
\coCart_{\bDelta^\op}
\end{tikzcd} \]
furnishes a morphism
\[
\wBSV_{|\D}
\longra
\rho_* \left( \rho^* \left( \wBSV_{|\D} \right) \right)
\simeq
\rho_* \BV
=:
\int_{|\D} \BV
\]
in $\coCart_\D$, which we will show is an equivalence.  It suffices to check this on fibers over an object $D \in \D$, i.e., that the functor
\begin{equation}
\label{must.show.equivce.from.wBSV.to.int.BV.over.R}
\cV^{\times D^{(1)}}
\simeq
\wBSV_{|D}
\longra
\int_D \BV
\end{equation}
is an equivalence.

There is an evident enrichment of $\coCart_{\bDelta^\op}$ over $\Cat_1$, determined by the formula
\[
\hom_\Cat
\left(
[n]
,
\ulhom_{\coCart_{\bDelta^\op}} ( \cY,\cY' )
\right)
\simeq
\hom_{\coCart_{\bDelta^\op}} ( [n] \times \cY,\cY' )
\]
(where we write $[n] \in \coCart_\bDelta^\op \simeq \Fun(\bDelta^\op,\Cat)$ for the constant diagram at the object $[n] \in \Cat$).  Using this enrichment, and writing
\[
\fC(D)
:=
\hom_\D(D,\rho{\bullet})
\in
\Fun(\bDelta^\op,\Cat)
\simeq
\coCart_{\bDelta^\op}
\]
for the value of the restricted (covariant) Yoneda embedding at the object $D \in \D$, the end formula for right Kan extensions yields an equivalence
\begin{equation}
\label{use.end.formula.for.rkan.to.identify.int.R.BV}
\int_D \BV
\simeq
\ulhom_{\coCart_{\bDelta^\op}} ( \fC(D) , \BV )
~.
\end{equation}

Note that the object $\fC(D) \in \coCart_{\bDelta^\op}$ is evidently free: it's in the image of the left Kan extension
\begin{equation}
\label{lkan.adjn.from.delta.op.leq.one.to.delta.op}
\begin{tikzcd}[column sep=1.5cm]
\coCart_{\bDelta_{\leq 1}^\op}
\arrow[dashed, hook, transform canvas={yshift=0.9ex}]{r}
\arrow[leftarrow, transform canvas={yshift=-0.9ex}]{r}[yshift=-0.2ex]{\bot}
&
\coCart_{\bDelta^\op}
\end{tikzcd}
\end{equation}
along the full inclusion $\bDelta^\op_{\leq 1} \hookra \bDelta^\op$.  Hence, we obtain a composite equivalence
\begin{align}
\nonumber
\int_D
\BV
& \underset{\Cref{use.end.formula.for.rkan.to.identify.int.R.BV}}{\simeq}
\ulhom_{\coCart_{\bDelta^\op}}(\fC(D),\BV)
\\
\label{use.lkan.adjn.from.delta.op.leq.one.to.delta.op}
& \simeq
\ulhom_{\coCart_{\bDelta^\op_{\leq 1}}} \left( \fC(D)_{|\bDelta^\op_{\leq 1}} , \BV_{|\bDelta^\op_{\leq 1}} \right)
\\
\label{reduce.to.just.one.since.BV.is.pt.at.zero}
& \simeq
\Fun \left( \fC(D)_{|[1]^\circ} , \BV_{|[1]^\circ} \right)
\\
\nonumber
& \simeq
\cV^{\times D^{(1)}}
~,
\end{align}
where equivalence \Cref{use.lkan.adjn.from.delta.op.leq.one.to.delta.op} follows from the (evidently $\Cat$-enriched) adjunction \Cref{lkan.adjn.from.delta.op.leq.one.to.delta.op} and equivalence \Cref{reduce.to.just.one.since.BV.is.pt.at.zero} follows from the fact that $\BV_{|[0]^\circ} \simeq \pt$.  By construction, it is now clear that the functor \Cref{must.show.equivce.from.wBSV.to.int.BV.over.R} is indeed an equivalence: the identity functor on the $\infty$-category $\cV^{\times D^{(1)}}$.
\end{proof}

\subsection{Cartesian enriched factorization homology}
\label{fact.hlgy.with.cartesian.enr}

\begin{construction}
\label{product.everything.over.int.BV.over.D}
Let $\cV$ be a cartesian symmetric monoidal $\infty$-category.  By \Cref{all.about.extended.cart.s.m.deloop}, there exists a canonical (\textit{extended}) \textit{total cartesian product} composite
\[ \begin{tikzcd}[column sep=1.5cm]
\dint_{\! \! \! |\D} \BV
\arrow{r}
\arrow[bend left, dashed]{rr}{\prod}
&
\wBSV
\arrow{r}{\prod}
&
\cV
\end{tikzcd}~. \]
\end{construction}

\begin{notation}
We write
\[
\D/\M
:=
\lim \left(
\begin{tikzcd}
&
\Ar(\M)
\arrow{d}{t}
\\
\D
\arrow[hook]{r}
&
\M
\end{tikzcd}
\right) \]
for the pullback.  This comes equipped with a functor
\[
\D/\M
\xlongra{s}
\D
\]
through which we obtain the pulled back cocartesian fibration
\[ \begin{tikzcd}
\dint_{\! \! \! |\D/\M} \cdiC
\arrow{r}
\arrow{d}
&
\dint_{\! \! \! |\D} \cdiC
\arrow{d}
\\
\D/\M
\arrow{r}[swap]{s}
&
\D
\end{tikzcd}~, \]
as well as a cocartesian fibration
\[ \begin{tikzcd}
\D / \M
\arrow{d}{t}
\\
\M
\end{tikzcd} \]
whose cocartesian monodromy functors are given by postcomposition.
\end{notation}

\begin{definition}
\label{definition.fully.functorial.enr.fact.hlgy}
Suppose that $(\cV,\times)$ is a cartesian symmetric monoidal $\infty$-category, and let $\cC \in \fCat_1(\cV)$ be a flagged $\cV$-enriched $(\infty,1)$-category.  Then, the (\bit{cartesian enriched}) \bit{factorization homology functor} of $\cC$ is the left Kan extension
\[ \begin{tikzcd}[column sep=2cm, row sep=1.5cm]
\dint_{\! \! \! |\D/\M} \cdiC
\arrow{r}{\int_{|\D/\M} \ulhom_\cC}
\arrow{d}
&
\dint_{\! \! \! |\D/\M} \BVx
\arrow{r}{\prod}
&
\cV
\\
\D/\M
\arrow{d}[swap]{t}
\\
\M
\arrow[dashed]{rruu}[swap]{\int_{(-)} \cC}
\end{tikzcd}~, \]
a fiberwise colimit since it is along a cocartesian fibration; assembling over all $\cC \in \fCat_1(\cV)$ determines a bifunctor
\[
\int_{(-)} (-)
:
\M \times \fCat_1(\cV)
\longra
\cV
~.
\]
\end{definition}

The notation $\int_{(-)} \cC$ of \Cref{definition.fully.functorial.enr.fact.hlgy} is justified by the following result.

\begin{remark}
\label{remark.cartesian.enr.fact.hlgy.recovers.AFR}
Under certain mild hypotheses on the cartesian symmetric monoidal $\infty$-category $(\cV,\times)$ (which in particular induce a canonical embedding $\Spaces \hookra \cV$), cartesian enriched factorization homology of $\cV$-enriched $(\infty,1)$-categories (\Cref{definition.fully.functorial.enr.fact.hlgy}) is equivalent to cartesian factorization homology of category-objects in $\cV$ (\Cref{definition.cartesian.fact.hlgy}).  The proof of this assertion amounts to an elaboration of the result that flagged $\Spaces$-enriched $(\infty,1)$-categories are equivalent to Segal spaces \cite[Theorem 4.4.7]{GH-enr}, or really its generalization \cite[Theorem 7.5]{Haug-rect}. 
\end{remark}

\begin{remark}
\label{remark.aug.cycl.cocycl.bialg.in.Cat}
As indicated in \Cref{rmk.bicomm.bialg.in.Cat}, to obtain the functoriality of enriched factorization homology over all of $\M$ it should suffice to assume that $\cV$ is an augmented bicommutative bialgebra object in $\Cat$.  Moreover, echoing \Cref{rmk.only.cyclically.monoidal}, if we are only interested in extending the definition of enriched factorization homology over some subcategory of $\M$, it can be possible to weaken this hypothesis further; for instance, if we are only interested extending the definition of enriched factorization homology over the full subcategory $\BW \subset \M$ on the smooth circle, then it should suffice to assume that $\cV$ is an augmented cyclically monoidal and cyclically comonoidal bialgebra object in $\Cat$.\footnote{To make this assertion true, it suffices to declare that it is true by definition.}
\end{remark}

\appendix

\section{Factorization systems}
\label{section.fact.syst}

The purpose of this section is to prove the following result.

\begin{prop}\label{freely.add.cocart.for.second.factor}
Let $\cB$ be an $\infty$-category equipped with a factorization system $[\cB_0;\cB_1]$.  There exists a left adjoint
\[ \begin{tikzcd}[column sep=2cm]
\Cat^{\cB_0}_{\cocart/\cB}
\arrow[dashed, transform canvas={yshift=0.9ex}]{r}
\arrow[\surjmonoleft, transform canvas={yshift=-0.9ex}]{r}[yshift=-0.2ex]{\bot}
&
\coCart_\cB
\end{tikzcd} \]
to the surjective monomorphism, which takes an object $(\cE \ra \cB)$ to the horizontal composite in the diagram
\begin{equation}
\label{diagram.freely.add.cocart.for.second.factor}
\begin{tikzcd}
\cE
\underset{\cB}{\times}
\Ar^{\cB_1}(\cB)
\arrow{r}
\arrow{d}
&
\Ar^{\cB_1}(\cB)
\arrow{r}{\ev_t}
\arrow{d}{\ev_s}
&
\cB
\\
\cE
\arrow{r}
&
\cB
\end{tikzcd}~.
\end{equation}
\end{prop}

The proof of \Cref{freely.add.cocart.for.second.factor} requires the following easy result.

\begin{lemma}\label{fact.syst.gives.restricted.cocart.fibn}
Let $\cB$ be an $\infty$-category equipped with a factorization system $[\cB_0;\cB_1]$.  Then the restricted evaluation functor
\begin{equation}
\label{restricted.ev.t.cocart.fibn}
\ev_t :
\Ar^{\cB_1}(\cB)
\longra
\cB
\end{equation}
is a cocartesian fibration.  Moreover, given an object $(\tilde{b} \ra b) \in \Ar^{\cB_1}(\cB)$ and a morphism $b \ra b'$ in $\cB$, a cocartesian lift is given by the commutative square
\begin{equation}
\label{cocart.lift.in.restricted.ev.t.cocart.fibn}
\begin{tikzcd}
\tilde{b}
\arrow[dashed]{r}{\cB_0}
\arrow{d}[swap]{\cB_1}
&
\tilde{b}'
\arrow[dashed]{d}{\cB_1}
\\
b
\arrow{r}
&
b'
\end{tikzcd}
\end{equation}
involving the unique indicated factorization of the composite $\tilde{b} \ra b \ra b'$ (thought of as a morphism in $\Ar^{\cB_1}(\cB)$ by reading horizontally).
\end{lemma}

\begin{proof}
First of all, the evaluation functor
\[
\Ar(\cB)
\xra{\ev_t}
\cB
\]
is a cocartesian fibration, with cocartesian morphisms those which become equivalences under the functor $\ev_s$.  Then, observe that the fully faithful inclusion
\[
\Ar(\cB)
\longhookla
\Ar^{\cB_1}(\cB)
\]
admits a left adjoint (given by taking a morphism $\varphi$ with canonical factorization $[\varphi_1;\varphi_2]$ to the morphism $\varphi_2$), and moreover that any morphism which this left adjoint takes to an equivalence is also taken to an equivalence by $\Ar(\cB) \xra{\ev_t} \cB$.  Thus, the claim follows from \cite[Lemma 2.2.4.11 and Remark 2.2.4.12]{LurieHA}.
\end{proof}

\begin{proof}[Proof of \Cref{freely.add.cocart.for.second.factor}]
First of all, the horizontal composite in \Cref{diagram.freely.add.cocart.for.second.factor} is indeed a cocartesian fibration by \Cref{fact.syst.gives.restricted.cocart.fibn}, since cocartesian fibrations are stable under pullback and composition.  So this construction defines a functor
\[
\Cat^{\cB_0}_{\cocart/\cB}
\longra
\Cat_{\cocart/\cB}~.
\]

Next, to see that this functor actually factors through the subcategory
\[
\coCart_\cB
\subset
\Cat_{\cocart/\cB}~,
\]
we must show that given a morphism
\begin{equation}
\cE_1
\longra
\cE_2
\end{equation}
in $\Cat^{\cB_0}_{\cocart/\cB}$, the horizontal functor $F$ in the resulting commutative diagram
\begin{equation}
\label{upside.down.coat.hanger}
\begin{tikzcd}[column sep=0cm]
\cE_1
\underset{\cB}{\times}
\Ar^{\cB_1}(\cB)
\arrow{rr}{F}
\arrow{rd}
&
&
\cE_2
\underset{\cB}{\times}
\Ar^{\cB_1}(\cB)
\arrow{ld}
\\
&
\Ar^{\cB_1}(\cB)
\arrow{d}[swap]{\ev_t}
\\
&
\cB
\end{tikzcd}
\end{equation}
preserves cocartesian lifts of morphisms in $\cB$.  So, fix a morphism $b \xra{\varphi} b'$ in $\cB$.  Suppose we are given a lift of its source $b \in \cB$ in
\[
\cE_1
\underset{\cB}{\times}
\Ar^{\cB_1}(\cB)~:
\]
this consists of an object $(\tilde{b} \ra b) \in \Ar^{\cB_1}(\cB)$ as well as an object $e \in \cE_1$ lying over $\tilde{b} \in \cB$.  Then, by \Cref{fact.syst.gives.restricted.cocart.fibn} and parts (2) and (3) of \cite[Proposition 2.4.1.3]{LurieHTT}, we can obtain a cocartesian lift of the morphism $b \ra b'$ by first extracting the unique factorization
\[
\begin{tikzcd}
\tilde{b}
\arrow[dashed]{r}{\cB_0}
\arrow{d}[swap]{\cB_1}
&
\tilde{b}'
\arrow[dashed]{d}{\cB_1}
\\
b
\arrow{r}
&
b'
\end{tikzcd}
\]
and then extracting a cocartesian lift of the morphism $\tilde{b} \ra \tilde{b}'$ in $\cB$ at $e \in \cE_1$.  So clearly the functor $F$ preserves cocartesian lifts over morphisms in the subcategory $\cB_0 \subset \cB$.  On the other hand, if the morphism $\varphi$ instead lies in the subcategory $\cB_1 \subset \cB$, then the map $\tilde{b} \ra \tilde{b}'$ will be an equivalence, and hence the cocartesian lift in $\cE_1$ will be an equivalence as well.  So the functor $F$ preserves cocartesian lifts over morphisms in the subcategory $\cB_1 \subset \cB$ as well.  Since cocartesian lifts in a cocartesian fibration compose, the existence of the factorization system $[\cB_0;\cB_1]$ on $\cB$ implies that the functor $F$ preserves all cocartesian lifts.  So our construction does indeed define a functor
\[
\Cat^{\cB_0}_{\cocart/\cB}
\longra
\coCart_\cB~.
\]

Thus, it remains to show that this functor is indeed a left adjoint.  For this, we first observe that the identity section of the functor $\ev_s$ in the diagram \Cref{diagram.freely.add.cocart.for.second.factor} induces a section
\[
\begin{tikzcd}
\cE
\underset{\cB}{\times}
\Ar^{\cB_1}(\cB)
\arrow{d}
\\
\cE
\arrow[dashed, bend left]{u}
\end{tikzcd}~,
\]
which lies in $\Cat^{\cB_0}_{\cocart/\cB}$ by our previous considerations.  We claim that if we consider $(\cE \ra \cB) \in \Cat^{\cB_0}_{\cocart/\cB}$ then this section will be the unit map, while if we consider $(\cE \ra \cB) \in \coCart_\cB$ then the downwards functor will be the counit map.  Indeed, if we are given any $(\cE \ra \cB) \in \Cat^{\cB_0}_{\cocart/\cB}$ and any $(\cD \ra \cB) \in \coCart_\cB$, it is straightforward to verify that these maps induce inverse equivalences
\[
\hom_{\Cat^{\cB_0}_{\cocart/\cB}}(\cE,\cD)
\simeq
\hom_{\coCart_\cB}
\left(
\cE
\underset{\cB}{\times}
\Ar^{\cB_1}(\cB)
,
\cD
\right)
\]
of spaces.
\end{proof}

\bibliographystyle{amsalpha}
\bibliography{Enriched}{}

\end{document}